\documentclass[reqno]{amsart}
\usepackage{amsmath,amssymb}
\usepackage{hyperref}

\newcommand*{\mailto}[1]{\href{mailto:#1}{\nolinkurl{#1}}}
\newcommand{\Skindef}{[\cdot,\cdot]}
 \newcommand{\Skdef}{(\cdot,\cdot)}
 \newcommand{\ov}{\overline}
 \newcommand{\wt}{\widetilde}

\newtheorem{theorem}{Theorem}[section]
\newtheorem{proposition}[theorem]{Proposition}
\newtheorem{definition}[theorem]{Definition}

\newtheorem{corollary}[theorem]{Corollary}
\theoremstyle{definition}
\newtheorem{example}[theorem]{Example}
\newtheorem{remark}[theorem]{Remark}
\newtheorem{hypothesis}[theorem]{Hypothesis}

\numberwithin{equation}{section}

\begin{document}

\title[Indefinite Sturm--Liouville Operators]{Perturbation and spectral theory for singular indefinite Sturm--Liouville operators}

\author[J. Behrndt]{Jussi Behrndt}
\address{Technische Universit\"{a}t Graz\\
Institut f\"ur Numerische Mathematik\\
Steyrergasse 30\\
8010 Graz, Austria}
\email{\mailto{behrndt@tugraz.at}}
\urladdr{\url{https://www.applied.math.tugraz.at/~behrndt/}}

\author[P. Schmitz]{Philipp Schmitz}
\address{Department of Mathematics\\
	Technische Universit\"at Ilmenau\\ Postfach 100565\\
	98648 Ilmenau\\ Germany}
\email{\mailto{philipp.schmitz@tu-ilmenau.de}}
\urladdr{\url{https://www.tu-ilmenau.de/de/analysis/team/philipp-schmitz/}}

\author[G. Teschl]{Gerald Teschl}
\address{Faculty of Mathematics\\ University of Vienna\\
Oskar-Morgenstern-Platz 1\\ 1090 Wien\\ Austria}
\email{\mailto{Gerald.Teschl@univie.ac.at}}
\urladdr{\url{https://www.mat.univie.ac.at/~gerald/}}

\author[C. Trunk]{Carsten Trunk}
\address{Department of Mathematics\\
Technische Universit\"at Ilmenau\\ Postfach 100565\\
98648 Ilmenau\\ Germany}
\email{\mailto{carsten.trunk@tu-ilmenau.de}}
\urladdr{\url{https://www.tu-ilmenau.de/de/analysis/team/carsten-trunk/}}


\begin{abstract} 
We study singular Sturm-Liouville operators of the form
\[
\frac{1}{r_j}\left(-\frac{\mathrm d}{\mathrm dx}p_j\frac{\mathrm d}{\mathrm dx}+q_j\right),\qquad j=0,1,
\]
in $L^2((a,b);r_j)$,
where, in contrast to the usual assumptions, the weight functions $r_j$ have different signs 
near the singular endpoints $a$ and $b$. In this situation the associated maximal operators become self-adjoint with respect to
indefinite inner products and their spectral properties differ essentially from the Hilbert space situation. 
We investigate the essential spectra and accumulation properties of nonreal and real discrete eigenvalues; we emphasize that  
here also perturbations of the indefinite weights $r_j$ are allowed. Special attention is paid to Kneser type results in the indefinite setting and to
$L^1$ perturbations of periodic operators. 
\end{abstract}

\maketitle

\section{Introduction}
The purpose of this paper is to develop perturbation methods and to study spectral properties of singular Sturm--Liouville
operators $K_0$ and $K_1$ associated with the differential expressions
\begin{equation}
\label{ell}
\ell_0=\frac{1}{r_0}\left(-\frac{\mathrm d}{\mathrm dx}p_0\frac{\mathrm d}{\mathrm dx}+q_0\right)
\quad \text{and}\quad
\ell_1=\frac{1}{r_1}\left(-\frac{\mathrm d}{\mathrm dx}p_1\frac{\mathrm d}{\mathrm dx}+q_1\right)
\end{equation}
on some interval $(a,b)$, where $-\infty\leq a<b\leq \infty$. As usual, we impose the standard
assumptions $1/p_j,q_j,r_j \in L^1_{\mathrm{loc}}(a,b)$ real, $r_j\neq 0$, $p_j>0$ a.\,e., and furthermore the endpoints $a$ and $b$ are assumed
to be singular and in the limit point case.
We will be interested in so-called {\it indefinite} Sturm--Liouville operators, i.e., we consider sign changing 
weight functions $r_j$; more precisely, here we treat the case $r_j<0$ in a neighbourhood of $a$
and $r_j>0$ in a neighbourhood of $b$, $j=0,1$. In this situation the maximal operators $K_j$, $j=0,1$, associated to $\ell_j$ in the weighted 
$L^2$-spaces $L^2((a,b);r_j)$ are self-adjoint with respect to the corresponding Krein space inner products 
\begin{equation*}
 [f,g]=\int_a^b f(x)\overline{g(x)} r_j(x)\, {\mathrm d}x,\qquad f,g\in L^2((a,b);r_j).
\end{equation*}
Various aspects in the spectral theory of indefinite Sturm-Liouville operators have been studied intensively in the mathematical literature 
and we refer the reader to
\cite{BCPQ14,BMT11,BMT14,BMT16,BST19,CQ14,CQX16,GSX17,KM16,LS16,M86,QX13} for different types of eigenvalue estimates and to
\cite{BE83,BM10,CN95,DL77,KK08,KKM09,KM07,KMWZ03,KWZ01,KWZ04,zettl} and the references therein for a discussion of so-called 
critical points, similarity, and special cases as, e.g., 
left definite problems.

A natural and intuitive approach to the spectral theory of indefinite Sturm-Liouville operators is to reduce a part of the analysis to the {\it definite}
case near the singular endpoints via Glazman's decomposition method and to apply perturbation techniques; this idea appears already in the 
fundamental paper \cite{CL89} and has been further applied and developed in, e.g., \cite{B07,BKT09,BP10,BPT13,BT07,KT09}.
More precisely, fix $a<\alpha<\beta<b$ such that both weights $r_j$ are negative on $(a,\alpha)$ and positive on $(\beta,b)$
and view the operators $K_j$ as finite rank perturbations in resolvent sense of the block diagonal operators
\begin{equation}\label{ortho}
H_j:=\begin{pmatrix} -H_{j,-} & 0 & 0 \\ 0 & K_{j, \alpha\beta} & 0\\ 0 & 0 & H_{j,+}\end{pmatrix}
\end{equation}
in the Krein spaces 
\begin{equation}\label{orthokkk}
L^2((a,b);r_j)= L^2((a,\alpha);r_j) \oplus L^2((\alpha,\beta);r_j) \oplus L^2((\beta,b);r_j).
\end{equation}
Note that sign changes of the weight functions are possible only inside the finite interval $[\alpha,\beta]$.
Hence
$K_{j, \alpha\beta}$ are regular indefinite Sturm-Liouville operators associated to $\ell_j$ in the Krein spaces $L^2((\alpha,\beta);r_j)$, 
and $H_{j,\pm}$ are definite singular Sturm-Liouville operators associated with the differential expressions
\begin{equation}
\label{tau}
\tau_j=\frac{1}{\vert r_j\vert}\left(-\frac{\mathrm d}{\mathrm dx}p_j\frac{\mathrm d}{\mathrm dx}+q_j\right),\qquad j=0,1,
\end{equation}
in the weighted $L^2$-Hilbert spaces $L^2((a,\alpha);-r_j)$ and $L^2((\beta,b);r_j)$. By imposing Dirichlet boundary conditions at
$\alpha$ and $\beta$ the operators $H_{j,\pm}$ become self-adjoint in the corresponding Hilbert spaces and $K_{j, \alpha\beta}$ becomes 
self-adjoint in the Krein space $L^2((\alpha,\beta);r_j)$. Due to the diagonal form it is clear that the spectra of $H_j$ coincide with the union of the spectra 
of the diagonal entries.
It is well-known that the spectrum of $K_{j, \alpha\beta}$ is purely discrete 
and hence the essential spectrum of the block diagonal operators $H_j$ in \eqref{ortho} is given by the union of the essential spectra of the 
Hilbert space self-adjoint operators $\pm H_{j,\pm}$. 

In order to conclude spectral properties of the operators $K_j$ 
from \eqref{ortho} a careful analysis of the underlying rank two perturbation (the functions in the domains of $H_j$ satisfy Dirichlet boundary conditions
at $\alpha$ and $\beta$) is necessary, which is particularly subtle due to the indefinite 
nature of the problem as self-adjoint operators in Krein spaces may have a rather arbitrary spectral structure. 
A priori it is not even clear if the resolvent sets of $K_j$ are nonempty, even rank one perturbations may lead to nonreal eigenvalues accumulating towards the essential spectrum, and other spectral effects can appear. 

An additional substantial difficulty when comparing the indefinite Sturm-Liouville
operators $K_0$ and $K_1$ stems from the fact that the operators act in different Krein spaces (as the weight functions 
$r_0$ and $r_1$ are different in general) and hence $H_{0,\pm}$ and $H_{1,\pm}$ act in different Hilbert spaces; at the same time
also the regular indefinite Sturm-Liouville operators 
$K_{0, \alpha\beta}$ and $K_{1, \alpha\beta}$ act in different Krein spaces.
We emphasize that perturbation theory for Sturm-Liouville operators with different weights $r_0\not=r_1$ has not obtained much 
attention and to the best of our knowledge there is only the contribution \cite{BSTT23} for the definite case that contains (nontrivial) results on the invariance of the essential spectrum. 

In this paper we use recent results from perturbation theory of definite Sturm-Liouville operators from our publications 
\cite{BSTT23,BSTT24}
and abstract perturbation results for self-adjoint operators in Krein spaces from \cite{B07-2,BMT14,BMT16} 
together with the above methodology to obtain a number of spectral results for singular indefinite Sturm-Liouville operators.
In Theorem~\ref{Eisenach} we provide conditions on the coefficients $r_j,p_j,q_j$ such that the essential spectra of $K_0$ and $K_1$ 
coincide, 
\begin{equation*}
 \sigma_\mathrm{ess}(K_0)=\sigma_\mathrm{ess}(K_1),
\end{equation*}
which can be considered as one of the main results.
In Corollary~\ref{Erfurt} we illustrate this general result in a more explicit situation, where it is assumed that $(a,b)=\mathbb R$ 
and the coefficients admit limits at the singular endpoints $\pm\infty$. 

The accumulation of nonreal eigenvalues towards certain regions of the real axis and the accumulation 
of discrete real eigenvalues towards the essential spectrum is investigated in Section~\ref{section3}. In Theorem~\ref{EisenachMirror2} and
Theorem~\ref{EisenachMirror3} these problems are treated for $K_0$ in a general setting in terms of the essential spectra of the operators $H_{0,\pm}$;
the concept behind is the so-called local definitizability of self-adjoint operators in Krein spaces and the stability of this property under 
finite rank perturbations in resolvent sense; cf. \cite{B07-2,BMT14,J88,J91,J03}. We pay special attention to the accumulation of real discrete eigenvalues in the case $(a,b)=\mathbb R$,
where the coefficients $r_j,p_j,q_j$ admit limits at $\pm\infty$ such that a gap arises in the essential spectra of $K_j$. This allows to 
conclude Kneser type results in the spirit of \cite{BSTT23,kt3,kt} in the indefinite setting in Theorem~\ref{Kneser} and Theorem~\ref{Kneser17}; 
cf. \cite{Kneser93} and also \cite{BMT11,DunfordSchwartz,GesztesyUnal98,H48,Hi48,ROFE-BEKETOV,Schmidt00,W12}. 

Another interesting situation appears in Section~\ref{section4} in the periodic setting under $L^1$-perturbations 
of the periodic coefficients of $K_0$: The band structure  
of the periodic operators $\pm H_{0,\pm}$ is preserved and leads to a band structure of the perturbed periodic operators $\pm H_{1,\pm}$,
and hence to a band structure of $K_1$; cf.\ Theorem~\ref{Lad}.
An additional finite first moment condition on the coefficient differences together with \cite[Theorem 2.3]{BSTT24} 
combined with the results in Section~\ref{section3} allows us to prove finiteness of 
eigenvalues in the spectral band gaps of the perturbed periodic operator $K_1$ outside a certain compact region; this can be viewed as an extension of a seminal result by Rofe-Beketov from the 
1960s to general indefinite Sturm-Liouville operators; cf. \cite{BrownEasthamSchmidt13,GS93,kt3,Schmidt00}. We also refer the reader to \cite{BMT11,DL86,K11,P13} for other related studies on indefinite 
Sturm-Liouville operators in the periodic setting.

For the convenience of the reader the paper contains a short appendix on operators in Krein spaces, where some spectral properties and perturbation results for 
self-adjoint operators with finitely many negative squares and locally definitizable self-adjoint operators from the mathematical literature are recalled.

Throughout the paper we shall use the notions of essential and discrete spectrum for operators that are not necessarily self-adjoint in a Hilbert space.
To avoid possible confusion we recall that
for a closed operator $T$ in a Hilbert space  $\lambda\in\mathbb C$ is a \textit{discrete eigenvalue} if $\lambda\in\sigma_\mathrm{p}(T)$ is an isolated point 
in the spectrum $\sigma(T)$ and the corresponding Riesz projection is a finite rank operator. The \textit{essential spectrum} $\sigma_\mathrm{ess}(T)$
is the complement of the discrete eigenvalues in $\sigma(T)$. We emphasize that under our assumptions 
the essential spectra of the Sturm-Liouville operators appearing in this paper is automatically real and remains invariant under compact perturbations in resolvent sense.
\\
\\
\noindent\textbf{Acknowledgements}.
J.B. gratefully acknowledges financial support by the Austrian Science Fund (FWF): P 33568-N. J.B. is also most grateful 
for the stimulating research
stay and the hospitality at the University of Auckland in February and March 2023, where some
parts of this paper were written. This publication is also based upon work from COST Action CA 18232 MAT-DYN-NET, supported by COST (European Cooperation in Science and Technology), www.cost.eu.

\section{Essential Spectrum}

Let $-\infty\leq a < b \leq \infty$ and let $\ell_j$ and $\tau_j$,
$j=0,1$, be the Sturm--Liouville expressions on $(a,b)$
in \eqref{ell} and \eqref{tau}, respectively, and assume that the
coefficients satisfy the
standard assumptions $1/p_j,q_j,r_j \in L^1_{\mathrm{loc}}(a,b)$ real,
$r_j\neq 0$, and $p_j>0$ almost everywhere. The next hypothesis
on the different signs of the weight functions near $a$ and $b$ is
central for the present paper.

\begin{hypothesis}\label{h1}
There exist $\alpha,\beta\in\mathbb R$ with $a<\alpha\leq\beta< b$ such that
  $r_j<0$ on $(a,\alpha)$ and $r_j>0$ on $(\beta,b)$
for $j=0,1$.
\end{hypothesis}

The Hilbert spaces of measurable complex valued functions $f$ on $(a,b)$
such that $r_jf^2\in L^1(a,b)$ are denoted by
$L^2((a,b); \lvert r_j\rvert)$ and are equipped with the standard
scalar products
\begin{equation*}
(f,g)= \int_a^b f(x)\overline{g(x)}\lvert r_j(x)\rvert\, \mathrm
dx,\qquad f,g\in L^2((a,b); \lvert r_j\rvert).
\end{equation*}
Besides these scalar products we shall also consider the inner products
\begin{equation}\label{indef}
[f,g]=\int_a^b f(x)\overline{g(x)}r_j(x)\, \mathrm dx,\qquad f,g\in
L^2((a,b); \lvert r_j\rvert),
\end{equation}
which are both indefinite by Hypothesis~\ref{h1}. The spaces $L^2((a,b);
\lvert r_j\rvert)$ equipped with $[\cdot,\cdot]$
become Krein spaces and will be denoted by $L^2((a,b); r_j)$. The
corresponding fundamental symmetries $J=\operatorname{sgn}(r_j)$ connect
the inner products via $[f,g]=(Jf,g)$.
Note also that the differential expressions $\ell_j$ are formally
symmetric with
respect to the indefinite inner products $[\cdot,\cdot]$ and that $\ell_j=J \tau_j$,
$j=0,1$.
For the basic properties of indefinite inner product spaces and
linear operators therein we refer to \cite{AI,B,Gheo22}.

Next we will define various Sturm-Liouville operators associated with
$\ell_j$ and $\tau_j$. For our purposes the following
hypothesis on the definite Sturm-Liouville expression $\tau_0$ is
appropriate.
\begin{hypothesis}\label{h2}
The endpoints $a$ and $b$ are singular and in the limit point case with
respect to $\tau_0$.
\end{hypothesis}\noindent
Let $j=0,1$ and denote the maximal domains by
\begin{align}\label{kondzilla}
        \mathcal D_j(a,b) = \left\{ f\in L^2((a,b); \lvert r_j\rvert) : f,
pf'\in \mathcal{AC}(a,b), \tau_j f\in L^2((a,b); \lvert r_j\rvert) \right\},
\end{align}
where $\mathcal{AC}(a,b)$
stands for the space of absolutely continuous functions on $(a,b)$. Note
that $\tau_j f\in L^2((a,b); \lvert r_j\rvert)$
if and only if $\ell_j f\in L^2((a,b); \lvert r_j\rvert)$. The maximal
operators associated to $\ell_j$ and $\tau_j$ are defined
as
\begin{equation}\label{maxis}
\begin{split}
  K_j f&=\ell_j f\qquad\operatorname{dom} K_j= \mathcal D_j(a,b),\\
  L_j g&=\tau_j g\qquad\, \operatorname{dom} L_j= \mathcal D_j(a,b),
  \end{split}
\end{equation}
for $j=0,1$.
Observe that Hypothesis~\ref{h2} ensures that the definite
Sturm-Liouville operator $L_0$  is self-adjoint  in the Hilbert space
$L^2((a,b); \lvert r_0\rvert)$ and from
$K_0=J L_0$ it follows that the indefinite Sturm-Liouville operator
$K_0$  is self-adjoint  in the Krein space $L^2((a,b);  r_0)$; cf.\ Appendix~\ref{appendix}.
Besides the natural maximal operators in \eqref{maxis} we shall also
make use of Sturm-Liouville operators associated to $\ell_j$ and
$\tau_j$ on the subintervals $(a,\alpha)$, $(\alpha,\beta)$, and
$(\beta,b)$ equipped with Dirichlet boundary conditions at the
regular endpoints $\alpha$ and $\beta$. More precisely, for $j=0,1$ we
define the operators
\begin{equation}\label{teile}
  \begin{split}
   H_{j,-}f&=\tau_j f,\qquad \operatorname{dom}
H_{j,-}=\bigl\{f\in\mathcal D_j(a,\alpha):f(\alpha)=0 \bigr\},\\
   K_{j, \alpha\beta} g&=\ell_j g,\qquad
        \operatorname{dom} K_{j, \alpha\beta}=\bigl\{g\in\mathcal
D_j(\alpha,\beta):g(\alpha)=g(\beta)=0 \bigr\},\\
   H_{j,+}h&=\tau_j h,\qquad \operatorname{dom}
H_{j,+}=\bigl\{h\in\mathcal D_j(\beta,b):h(\beta)=0 \bigr\},
  \end{split}
\end{equation}
where the maximal domains $\mathcal D_j(a,\alpha)$, $\mathcal D_j(\alpha,\beta)$, and $\mathcal D_j(\beta,b)$ 
are defined in the same way as $\mathcal D_j(a,b)$ in \eqref{kondzilla}.
It follows from Hypothesis~\ref{h2} that the operators $H_{j,-}$ are
self-adjoint in the Hilbert spaces $L^2((a,\alpha); \vert r_j\vert)$
and the operators $H_{j,+}$ are self-adjoint in the Hilbert spaces
$L^2((\beta,b); \vert r_j\vert)$. The next hypothesis ensures that the operators $L_0$ and $H_{0,\pm}$ are semibounded from below 
(see, e.g. \cite[Proof of Theorem 3.2]{BSTT23}).
\begin{hypothesis}\label{h3}
The function $q_0/r_0$ is bounded near $a$ and $b$.
\end{hypothesis}\noindent
We emphasize that the regular indefinite Sturm-Liouville operators $K_{j, \alpha\beta}$
are self-adjoint in $L^2((\alpha,\beta); r_j)$ (see, e.g., \cite{CL89}), where these spaces
are equipped with the inner product \eqref{indef} restricted to
$(\alpha,\beta)$. 
\begin{remark}\label{remarkspecial}
Note that the inner product \eqref{indef} $(\alpha,\beta)$ can be indefinite or definite, depending
on the properties of the weight functions $r_j$, $j=0,1$, on $(\alpha,\beta)$. 
To
avoid confusion we will always view $L^2((\alpha,\beta); r_j)$ as a
Krein space (which reduces to a
Hilbert space in the special case $r_j>0$ a.e. on $(\alpha,\beta)$). Furthermore, the case $\alpha=\beta$ in Hypothesis~\ref{h1} is understood in the 
sense that the regular indefinite Sturm-Liouville operators $K_{j, \alpha\beta}$ and the Krein spaces $L^2((\alpha,\beta); r_j)$ are absent in the
operator and space decompositions appearing later in the text. 
\end{remark}

The next result is the first main result in this paper. We provide criteria on the coefficients of $\ell_0$ and
$\ell_1$ such that the essential spectra of the indefinite
Sturm-Liouville operators $K_0$ and $K_1$ coincide. It also turns out
that the assumptions on the coefficients imply that the
resolvent sets of $K_0$ and $K_1$ are both nonempty.

\begin{theorem}\label{Eisenach}
Assume Hypotheses~\ref{h1}, \ref{h2}, and \ref{h3}, and suppose that for each
endpoint $e\in\{a,b\}$ the following conditions hold:
$$\lim\limits_{x\rightarrow e} \frac{r_1(x)}{r_0(x)}=1,\quad
\lim\limits_{x\rightarrow e} \frac{p_1(x)}{p_0(x)}=1,\quad
\lim\limits_{x\rightarrow e} \frac{q_1(x)-q_0(x)}{r_0(x)}=0.$$
Then both indefinite Sturm-Liouville operators $K_0$ and $K_1$ are
self-adjoint in the Krein spaces $L^2((a,b);r_0)$ and $L^2((a,b);r_1)$,
respectively, the
resolvent sets $\rho(K_0)$ and $\rho(K_1)$ are nonempty, and
\begin{equation}\label{essi3+}
        \sigma_\mathrm{ess}(K_{0}) =\sigma_\mathrm{ess}(K_{1})\subset \mathbb R.
\end{equation}
\end{theorem}

\begin{proof}
In the present situation it follows from \cite[Theorem 3.2]{BSTT23} that
both singular endpoints $a$ and $b$ are also in the limit
point case with respect to $\tau_1$ and hence the maximal definite
Sturm-Liouville operator $L_1$ associated to $\tau_1$ is
self-adjoint in $L^2((a,b);\vert r_1\vert)$. It also follows from Hypothesis~\ref{h3} and one more application of
\cite[Theorem 3.2]{BSTT23}
that the operators $H_{j,\pm}$ are self-adjoint and semibounded in the
corresponding
Hilbert spaces $L^2((b,\beta);r_j)$ and $L^2((a,\alpha);  -r_j)$, and that
\begin{equation}\label{essi1}
\sigma_\mathrm{ess}(H_{0,+})
=\sigma_\mathrm{ess}(H_{1,+})
\quad\text{and}\quad
\sigma_\mathrm{ess}(H_{0,-})
=\sigma_\mathrm{ess}(H_{1,-}).
\end{equation}
The semiboundedness of $H_{0,\pm}$ and $H_{1,\pm}$ also implies the
semiboundedness of $L_0$ and $L_1$, respectively.
Hence it follows from \cite[Theorem 4.5]{BP10} that
both indefinite Sturm-Liouville operators $K_0$ and $K_1$ are
self-adjoint in the Krein spaces $L^2((a,b);r_0)$ and $L^2((a,b);r_1)$,
respectively, and that the
resolvent sets $\rho(K_0)$ and $\rho(K_1)$ are nonempty.

We now turn to the essential spectra of $K_0$ and $K_1$ and verify the
remaining assertion \eqref{essi3+}. For this we first consider
the orthogonal sums $H_j$, $j=0,1$, given by 
\begin{equation}\label{ortho2}
H_0=\begin{pmatrix} -H_{0,-} & 0 & 0 \\ 0 & K_{0, \alpha\beta} & 0\\ 0
& 0 & H_{0,+}\end{pmatrix}\quad\text{and}\quad
H_1=\begin{pmatrix} -H_{1,-} & 0 & 0 \\ 0 & K_{1, \alpha\beta} & 0\\ 0
& 0 & H_{1,+}\end{pmatrix},
\end{equation}
on their natural domains 
$$\operatorname{dom} H_j=\operatorname{dom} H_{j,-}\times\operatorname{dom} K_{j,\alpha\beta}\times\operatorname{dom} H_{j,+}, \quad j=0,1,$$ 
in the Krein spaces \eqref{orthokkk}.
Note
that the operators $K_{0, \alpha\beta}$ and $K_{1, \alpha\beta}$ are
both regular indefinite Sturm-Liouville operators that are self-adjoint
in the Krein space
$L^2((\alpha,\beta); r_0)$ and $L^2((\alpha,\beta); r_1)$, respectively,
and that the spectra of these operators consist of real discrete
eigenvalues accumulating to $\infty$ and
$-\infty$. In addition, there may appear at most
finitely many pairs of nonreal discrete eigenvalues which are symmetric with
respect to the real line. This is a consequence of the fact
that $K_{0, \alpha\beta}$ and $K_{1, \alpha\beta}$ have finitely many negative squares; we refer the reader to \cite{CL89} and Theorem~\ref{Reklame} 
in Appendix~\ref{appendix}.

Therefore, we conclude that $H_0$ and $H_1$ are self-adjoint in the Krein spaces
$L^2((a,b);r_0)$ and $L^2((a,b);r_1)$, respectively, and
from \eqref{essi1} and \eqref{ortho2} we obtain that
the essential spectra of $H_0$ and $H_1$ coincide and
\begin{equation}\label{essi2}
\begin{split}
  \sigma_\mathrm{ess}(H_{0})
  &=\sigma_\mathrm{ess}(H_{0,+})\cup\sigma_\mathrm{ess}(-H_{0,-}) \\
  &=\sigma_\mathrm{ess}(H_{1,+})\cup\sigma_\mathrm{ess}(-H_{1,-})
  =\sigma_\mathrm{ess}(H_{1}).
  \end{split}
\end{equation}
As $\rho(K_0)$ and $\rho(K_1)$ are both nonempty we see that for
$\lambda\in\rho(K_0)\cap\rho(H_0)$ and $\mu\in\rho(K_1)\cap\rho(H_1)$
the resolvent differences
\begin{equation}\label{pertu}
  (K_0-\lambda)^{-1}-(H_0-\lambda)^{-1} \quad\text{and}\quad
(K_1-\mu)^{-1}-(H_1-\mu)^{-1}
\end{equation}
are rank two operators. In fact, this follows from the observation that
$K_j f= H_jf$ for all
$f\in\mathcal D_j(a,b)$ such that $f(\alpha)=f(\beta)=0$, and it is
clear that this subspace of functions
is a two-dimensional restriction of the maximal domains.
Hence the essential spectra of $K_0$ and $H_0$ coincide and the
essential spectra of $K_1$ and $H_1$ coincide, that is, 
\begin{equation}\label{wuerstchen}
 \sigma_\mathrm{ess}(K_{0})=\sigma_\mathrm{ess}(H_{0})\quad\text{and}\quad\sigma_\mathrm{ess}(K_{1})=\sigma_\mathrm{ess}(H_{1}).
\end{equation}
This observation together with \eqref{essi2} leads to \eqref{essi3+};
note that the essential spectrum is real as $H_{0,\pm}$ and $H_{1,\pm}$
are self-adjoint in Hilbert spaces.
\end{proof}

In the next corollary we use
Theorem~\ref{Eisenach} to express the essential spectrum in a more explicit way for the case where the
coefficients of $\ell_0$ and $\ell_1$
admit the same limits $r_{\pm\infty},p_{\pm\infty},q_{\pm\infty}$ at
$\pm\infty$. 

\begin{corollary}\label{Erfurt}
        Assume that $(a,b)=\mathbb R$ and that the coefficients $r_j,p_j,q_j$,
$j=0,1$, admit the limits 
        \begin{equation*}
        \begin{split}
                r_{\pm\infty}&=\lim_{x\rightarrow \pm\infty} r_0(x)=\lim_{x\rightarrow \pm\infty} r_1(x),\\
                p_{\pm\infty}&=\lim_{x\rightarrow \pm\infty} p_0(x)=\lim_{x\rightarrow \pm\infty} p_1(x),\\ 
                q_{\pm\infty}&=\lim_{x\rightarrow \pm\infty} q_0(x)=\lim_{x\rightarrow \pm\infty} q_1(x),
        \end{split}\end{equation*}
       where $\pm r_{\pm\infty}>0$, $p_{\pm\infty}>0$, and $q_{\pm\infty}\in\mathbb R$.
        Then both indefinite Sturm-Liouville operators $K_0$ and $K_1$ are
self-adjoint in the Krein spaces $L^2(\mathbb R;r_0)$ and $L^2(\mathbb
R;r_1)$, respectively, the
resolvent sets $\rho(K_0)$ and $\rho(K_1)$ are nonempty, and
\begin{equation}\label{essi3}
        \sigma_\mathrm{ess}(K_{0}) =\sigma_\mathrm{ess}(K_{1})=
\left(-\infty,-\frac{q_{-\infty}}{r_{-\infty}}\right] \cup
\left[\frac{q_{+\infty}}{r_{+\infty}},\infty\right).
\end{equation}
        In particular, there is a gap in the essential spectrum if
$-q_{-\infty}/r_{-\infty}<q_{+\infty}/r_{+\infty}$.
\end{corollary}
       
        \begin{proof}
        We shall define a suitable comparison operator $K$
with constant coefficients and apply Theorem~\ref{Eisenach} to the pairs
$\{K,K_0\}$ and $\{K,K_1\}$. 
        Consider the piecewise constant coefficients $r,p,q$ defined by
        \begin{align*}
                r(x) = \begin{cases}
                        r_{+\infty}&\text{if } x\geq 0,\\
                        r_{-\infty}&\text{if } x< 0,\\
                \end{cases}\quad
                p(x) = \begin{cases}
                p_{+\infty}&\text{if } x\geq 0,\\
                p_{-\infty}&\text{if } x< 0\\
                \end{cases},\quad q(x) = \begin{cases}
                q_{+\infty}&\text{if } x\geq 0,\\
                q_{-\infty}&\text{if } x< 0,\\
                \end{cases}
        \end{align*}
        and note that the corresponding definite Sturm-Liouville expression
        \begin{equation*}
\tau=\frac{1}{\vert r\vert}\left(-\frac{\mathrm d}{\mathrm
dx}p\frac{\mathrm d}{\mathrm dx}+q\right)
\end{equation*}
        is in the limit point case at both singular endpoints $\pm\infty$, that
is, Hypothesis~\ref{h2} holds. From the choice of $q$ and $r$, and the
assumptions on
        $r_0,r_1$ it is also clear that
        Hypothesis~\ref{h1} and Hypothesis~\ref{h3} are satisfied. Now we associate the operators
$K,L,H_+,K_{\alpha\beta},H_-$ to $\tau$ and its indefinite counterpart
        $\ell=\operatorname{sgn}(r)\tau$
        in the same way as in \eqref{kondzilla}, \eqref{maxis}, and
\eqref{teile}. Then the essential spectrum of the operators
        \begin{equation*}
  \begin{split}
   H_{-}f&=\tau f,\qquad \operatorname{dom} H_{-}=\bigl\{f\in\mathcal
D(a,\alpha):f(\alpha)=0 \bigr\},\\
   H_{+}h&=\tau h,\qquad \operatorname{dom} H_{+}=\bigl\{h\in\mathcal
D(\beta,b):h(\beta)=0 \bigr\},
  \end{split}
\end{equation*}
         is given by
        \begin{equation}\label{speckis}
                \sigma_{\mathrm{ess}}(H_{+}) =
\left[\frac{q_{+\infty}}{r_{+\infty}},\infty\right)\quad\text{and}\quad
                \sigma_{\mathrm{ess}}(H_{-}) =
\left[\frac{q_{-\infty}}{r_{-\infty}},\infty\right)
        \end{equation}
        (cf.\ \cite[Example on p.\ 209]{te}), and hence we conclude in the same
way as in \eqref{ortho2}, \eqref{essi2}, and \eqref{wuerstchen}  that
        \begin{equation}\label{essi4}
        \sigma_\mathrm{ess}(K)=\left(-\infty,-\frac{q_{-\infty}}{r_{-\infty}}\right] \cup \left[\frac{q_{+\infty}}{r_{+\infty}},\infty\right).
\end{equation}
        It is easy to see that the conditions
$$\lim\limits_{x\rightarrow \pm\infty}
\frac{r_j(x)}{r(x)}=1,\quad \lim\limits_{x\rightarrow \pm\infty}
\frac{p_j(x)}{p(x)}=1,\quad \lim\limits_{x\rightarrow \pm\infty}
\frac{q_j(x)-q(x)}{r(x)}=0$$
        hold for $j=0,1$ and hence
        we conclude from Theorem~\ref{Eisenach} that the indefinite
Sturm-Liouville operators $K_0$ and $K_1$
        are self-adjoint in the Krein spaces $L^2(\mathbb R;r_0)$ and
$L^2(\mathbb R;r_1)$, the
resolvent sets $\rho(K_0)$ and $\rho(K_1)$ are nonempty, and
\eqref{essi3} follows from \eqref{essi4} and \eqref{essi3+}.
\end{proof}

\section{Discrete spectrum}\label{section3}

In this section we study
the point spectrum of the indefinite Sturm-Liouville operator
$K_0$ from the previous section. Here we concentrate on the 
nonreal point spectrum and accumulation points in gaps 
of the essential spectrum. Recall from the previous section that under Hypotheses \ref{h1}, \ref{h2}, and \ref{h3} 
the essential spectrum of $K_0$ is given by 
\begin{equation}\label{pol}
	\sigma_\mathrm{ess}(K_0)=
	\sigma_\mathrm{ess}(H_{0,+}) \cup \sigma_\mathrm{ess}(-H_{0,-})
\end{equation}
and hence it is clear that the nonreal spectrum of $K_0$ consists 
of discrete eigenvalues that may only accumulate to points in 
\eqref{pol}. In the next theorem we observe that
accumulation of nonreal eigenvalues is possible only towards certain
subsets of \eqref{pol}.

\begin{theorem}\label{EisenachMirror2}
Assume Hypotheses~\ref{h1}, \ref{h2}, and \ref{h3}. Then the
following assertions hold.
\begin{itemize}
\item [{\rm (i)}] 
	The nonreal spectrum of $K_0$ consists 
of discrete  eigenvalues with geometric multiplicity one which are contained
in a compact subset of $\mathbb C$. 
\item [{\rm (ii)}]  
The nonreal spectrum 
may only accumulate to points in 
	$$
	\sigma_\mathrm{ess}(H_{0,+}) \cap \sigma_\mathrm{ess}(-H_{0,-}).
	$$
	Furthermore, the nonreal spectrum cannot accumulate to all 
	points $\lambda$ from the boundary $($in $\mathbb R)$
	\begin{equation}\label{Samoa}
	\partial \left(\sigma_\mathrm{ess}(H_{0,+}) \cap \sigma_\mathrm{ess}(-H_{0,-})\right)
	\end{equation} 
	with the following property $(${\rm P}$)$:
	There exists a left sided neighbourhood in $\mathbb R$ of $\lambda$ in 
	$\rho(H_{0,+})$ $(\rho(-H_{0,-}))$ and
	a right sided neighbourhood in $\mathbb R$ of $\lambda$ in 
	$\rho(-H_{0,-})$ $(\rho(H_{0,+})$, respectively$)$.
\item [{\rm (iii)}] In particular, if the interior $($in $\mathbb R)$
$$
	\mbox{\rm int}\left(\sigma_\mathrm{ess}(H_{0,+}) \cap \sigma_\mathrm{ess}(-H_{0,-})\right) =\emptyset
	$$
and all $\lambda$
in the set \eqref{Samoa} satisfy property $(${\rm P}$)$, then
 $K_0$ has at most finitely many nonreal eigenvalues.
\end{itemize}
\end{theorem}

\begin{proof}
We shall first show in Step 1 and Step 2 that the self-adjoint operator 
\begin{equation}\label{ortho33}
H_0=\begin{pmatrix} -H_{0,-} & 0 & 0 \\ 0 & K_{0,\alpha\beta} & 0\\ 0 & 0 & H_{0,+}\end{pmatrix}
\end{equation}
acting in the Krein space
$$
L^2((a,b);r_0)= L^2((a,\alpha);r_0) \oplus L^2((\alpha,\beta);r_0) \oplus L^2((\beta,b);r_0);
$$
is definitizable over the domain 
\begin{equation}\label{dommi}
\begin{split}
&\bigl(\overline{\mathbb C} \setminus (\sigma_\mathrm{ess}(H_{0,+}) \cap \sigma_\mathrm{ess}(-H_{0,-}))\bigr)\\
&\qquad\qquad\cup\bigl\{\lambda\in \partial \left(\sigma_\mathrm{ess}(H_{0,+}) \cap \sigma_\mathrm{ess}(-H_{0,-})\right): \lambda\text{ has property } ({\rm P}) \bigr\}
\end{split}
\end{equation}
in the sense of Definition~\ref{locdef} (see also \cite{J88,J03} and \cite{T15}). Using the perturbation result in Theorem~\ref{finite}~(i) 
we conclude in Step 3 
that the same is true for $K_0$, which implies the assertions (i)-(iii).
\vskip 0.2cm\noindent
\textit{Step 1.} We verify Definition~\ref{locdef}~(ii).
By Hypothesis~\ref{h1} $L^2((a,\alpha); r_0)$ 
is an anti-Hilbert space, $L^2((\beta,b); r_0)$
is a Hilbert space, and $L^2((\alpha, \beta); r_0)$ is 
(in general) a Krein space.
Since $H_{0,+}$ and $-H_{0,-}$
are self-adjoint in Hilbert spaces their spectra are real. 
In the present situation $\sigma(H_{0,+})$ is of positive type in the sense of Definition~\ref{definition++} and $\sigma(-H_{0,-})$ is of negative type
in the sense of Definition~\ref{definition++}.
As mentioned in the proof of Theorem~\ref{Eisenach}
the operator $K_{0,\alpha\beta}$ is a regular indeﬁnite Sturm-Liouville operator with finitely many negative squares 
(for the notion of negative squares we refer to the Appendix) and the
spectrum consists only of discrete eigenvalues; the real eigenvalues accumulate to 
$\infty$ and $-\infty$. 
Hence for every real $\lambda$
not in $\sigma_\mathrm{ess}(H_{0,+}) \cap \sigma_\mathrm{ess}(-H_{0,-})$
and  all points in~\eqref{Samoa} which satisfy 
property $(${\rm P}$)$ item (ii) from Definition~\ref{locdef} is satisfied for the operator $H_0$ in \eqref{ortho33}. 

We show next that 
item (ii) from Definition~\ref{locdef} is also satisfied for $\lambda=\infty$.
In fact, it follows from Hypothesis~\ref{h3} that  $H_{0,+}$ is semibounded from below and $-H_{0,-}$
is semibounded from above and hence  
$(-\infty,-\gamma) \subset \rho(H_{0,+})$ and $(\gamma,\infty) \subset \rho(-H_{0,-})$ for some $\gamma \geq 0$. As mentioned above the spectrum
of $H_{0,+}$ is of positive type and the spectrum
of $-H_{0,-}$ is of negative type. 
Since $K_{0,\alpha\beta}$ has finitely many negative squares Theorem~\ref{Reklame} (iii) implies that one can choose
$\gamma$ in such a way that $(-\infty,-\gamma)\cap\sigma(K_{0,\alpha\beta})$ is of negative type and 
$(\gamma,\infty)\cap\sigma(K_{0,\alpha\beta})$ is of positive type. Now the claim follows 
for $\lambda=\infty$.
\vskip 0.2cm\noindent
\textit{Step 2.} Observe that the nonreal spectrum of $H_0$ coincides with the nonreal spectrum of $K_{0,\alpha\beta}$ and hence it follows from 
Theorem~\ref{Reklame}~(i)  that the nonreal spectrum of $H_0$ consists of isolated points which
are poles of the resolvent of $H_0$ and
no point of $\mathbb R \cup \{\infty\}$ is an accumulation point
of non-real spectrum of $H_0$.
It remains to verify Definition~\ref{locdef}~(i). In fact, 
the growth condition for the resolvent of $H_0$ is valid since the resolvents of 
$H_{0,+}$ and $-H_{0,-}$ are bounded by $\vert$Im$\, \lambda \vert^{-1}$ for nonreal $\lambda$ and the operator 
$K_{0,\alpha\beta}$ satisfies the growth condition for the resolvent by Theorem~\ref{Reklame}~(iv). Hence $H_0$ is 
definitizable over the domain \eqref{dommi}.
\vskip 0.2cm\noindent
\textit{Step 3.}  By~\eqref{pertu}
the difference of the resolvents of $H_0$ and $K_0$ is a rank two operator and hence 
Theorem~\ref{finite} implies that the operator $K_0$ is also definitizable over the domain \eqref{dommi}; note 
that $\rho(K_0)\not=\emptyset$ follows from Hypothesis~\ref{h3} and \cite[Theorem~4.5]{BP10}.
This yields (ii) in Theorem~\ref{EisenachMirror2}.
Observe that the nonreal spectrum of $K_0$ is discrete by \eqref{pol} and  Hypothesis~\ref{h2} implies 
that all eigenvalues of $K_0$ have geometric multiplicity one. 
As $\infty$ is contained in the domain \eqref{dommi} it is also clear that
$\infty$ is not an accumulation point of nonreal eigenvalues of $K_0$.
This shows (i) and assertion (iii) is an immediate consequence
of (i) and (ii).
\end{proof}

Next we illustrate the decisive role of points with property (P) from Theorem~\ref{EisenachMirror2}~(ii) with some examples. 

\begin{example}\label{ex1}
Let $(a,b)=\mathbb R$ and consider the shifted Coulomb potential 
$$
r_0(x) = \mbox{sgn}\, x, \quad p_0(x)=1 \quad \mbox{and} \quad q_0(x)= 
-\frac{1}{1+\vert x \vert}, \quad x\in \mathbb R,
$$
where $\alpha=\beta=0$ in Hypothesis~\ref{h1}, see also \cite{BKT08,LS16} and \cite{vH78}. In this situation Corollary~\ref{Erfurt} and \eqref{speckis} 
show $\sigma_\mathrm{ess}(H_{0,+})=[0,\infty)$ and $\sigma_\mathrm{ess}(-H_{0,-})=(-\infty,0]$ and hence 
\eqref{Samoa} turns into
\begin{equation}\label{Ichwue}
\partial \left(\sigma_\mathrm{ess}(H_{0,+}) \cap \sigma_\mathrm{ess}(-H_{0,-})\right)
=\{0\}.
\end{equation}
Note that the operators $\pm H_{0, \pm}$ from \eqref{teile}
act in $L^2(\mathbb R_\pm)$ and 
the operator $K_{0, \alpha\beta}$
is not present in the decomposition \eqref{ortho}; cf. Remark~\ref{remarkspecial}.
 A Kneser type 
argument (see, e.g., \cite[Corollary~9.43]{te}) shows that
the discrete
spectrum in $(-\infty, 0)$ of the operator $H_{0,+}$ 
accumulates to zero from below and the discrete
spectrum in $(0,\infty)$ of the operator $-H_{0,-}$ 
accumulates to zero from above. Hence property (P) is not fulfilled
for the point zero and thus accumulation of nonreal eigenvalues of $K_0$ to zero may appear

In fact, it was shown in \cite{LS16} that for the 
shifted (indefinite) Coulomb potential the nonreal discrete
eigenvalues indeed accumulate to zero. 
In general it is an open problem formulated in \cite{B13}
whether the accumulation of discrete eigenvalues from below to the lower
boundary of the essential spectrum of $H_{0,+}$ or the accumulation of discrete eigenvalues from above to the upper
boundary of the essential spectrum of $-H_{0,-}$ leads to nonreal
accumulation of discrete eigenvalues to points from \eqref{Samoa}
not satisfying property (P).
\end{example}

\begin{example}
Let again  $(a,b)=\mathbb R$, $\alpha=\beta=0$, and consider a potential with fast decay towards $\pm\infty$
as in \cite{BKT09},
$$
r_0(x) = \mbox{sgn}\, x\quad p_0(x)=1 \quad \mbox{and} \quad q_0(x)= 
-\kappa(\kappa +1){\rm sech}^2(x), \quad x\in \mathbb R,
$$
for some $\kappa \in \mathbb N$. Then again the essential spectrum of $\pm H_{0,\pm}$ coincides
with $\mathbb R_\pm$ and \eqref{Ichwue} holds.
In contrast to Example~\ref{ex1} here a Kneser type argument shows that
the discrete spectra of $H_{0,+}$ in $(-\infty, 0)$  
and of $-H_{0,-}$  in $(0,\infty)$ are finite. Hence, zero has property (P) and Theorem \ref{EisenachMirror2}~(iii) shows that
the operator $K_0$ has finitely many nonreal eigenvalues; cf. \cite{BKT09}.
\end{example}

In the next theorem we complement the previous observations in Theorem~\ref{EisenachMirror2} and study the
possible accumulation of real discrete eigenvalues of $K_0$ to boundary points of \eqref{pol}. 
In this context we recall from \cite[Corollary 6]{H52} that $\sigma_\mathrm{ess}(\pm H_{0,\pm})$ can be any closed subset of $\mathbb R$ 
and, in particular, spectral points as in Theorem~\ref{EisenachMirror3}~(iii) below may appear.

\begin{theorem}\label{EisenachMirror3}
Assume Hypotheses~\ref{h1}, \ref{h2}, and \ref{h3}.
Then the following assertions hold.
\begin{itemize}
\item [{\rm (i)}] If $\lambda^+ \in\partial \sigma_\mathrm{ess}(H_{0,+})$
$($boundary in $\mathbb R)$ and $\lambda^+ \not\in \sigma_\mathrm{ess}(-H_{0,-})$, 
then the real discrete eigenvalues of $K_0$ accumulate
at $\lambda^+$ if and only if the discrete eigenvalues of $H_{0,+}$ accumulate
at $\lambda^+$.
\item [{\rm (ii)}]
If $\lambda^- \in\partial \sigma_\mathrm{ess}(-H_{0,-})$
$($boundary in $\mathbb R)$ and $\lambda^- \not\in \sigma_\mathrm{ess}(H_{0,+})$, then the real discrete eigenvalues of $K_0$ accumulate
at $\lambda^-$ if and only if the discrete eigenvalues of $-H_{0,-}$ accumulate
at $\lambda^-$.
\item [{\rm (iii)}] If $\lambda$ in \eqref{Samoa} satisfies $(${\rm P}$)$ and $\lambda\not\in \mbox{\rm int}(\sigma_\mathrm{ess}(K_0))$
$($interior in $\mathbb R)$, 
then the real discrete eigenvalues of $K_0$ accumulate
at $\lambda$.
\end{itemize}
\end{theorem}

\begin{proof}
(i) and (ii):
In the proof of Theorem~\ref{EisenachMirror2} it was shown
that $H_0$ and $K_0$ are definitizable over the domain \eqref{dommi}
in all points $\lambda\in \mathbb R \cup \{\infty\}$  not in
$\sigma_\mathrm{ess}(H_{0,+}) \cap \sigma_\mathrm{ess}(-H_{0,-})$
and in all points in~\eqref{Samoa} which satisfy 
property $(${\rm P}$)$. 
Then the statement on the accumulation of eigenvalues
 follows from Theorem~\ref{finite} (ii).

(iii) Assume that $\lambda$ in \eqref{Samoa} satisfies $(${\rm P}$)$ and $\lambda\not\in \mbox{\rm int}(\sigma_\mathrm{ess}(K_0))$. Consider the
case where a left sided neighbourhood in $\mathbb R$ of $\lambda$ is contained in 
	$\rho(H_{0,+})$ and
	a right sided neighbourhood in $\mathbb R$ of $\lambda$ is contained in
	$\rho(-H_{0,-})$. Since $\lambda\in\sigma_\mathrm{ess}(\pm H_{0,\pm})$ it follows 
	from the assumption $\lambda\not\in\mbox{\rm int}(\sigma_\mathrm{ess}(K_0))=\mbox{\rm int}(\sigma_\mathrm{ess}(H_{0,+}) \cup \sigma_\mathrm{ess}(-H_{0,-}))$
	that $\lambda$ is an isolated point in the set $\sigma_\mathrm{ess}(H_{0,+})$ or in the set $\sigma_\mathrm{ess}(-H_{0,-})$. In the present situation 
	it follows 
	that 
	$(\lambda,\lambda+\varepsilon)\cap\sigma(H_{0,+})$ or $(\lambda-\varepsilon,\lambda)\cap\sigma(-H_{0,-})$ consists of discrete eigenvalues that accumulate to $\lambda$. Now the statement follows again from Theorem~\ref{finite} (ii).
 \end{proof}

 \begin{remark}\label{rettung}
  Note that \eqref{pol}, Theorem~\ref{EisenachMirror2}, and Theorem~\ref{EisenachMirror3} remain true if instead of Hypothesis~\ref{h3} one assumes that 
  the operators $H_{0,+}$ and $H_{0,-}$ are both semibounded from below; equivalently one may assume that $L_0$ is semibounded from below.
 \end{remark}

We shall illustrate the statements in the Theorems~\ref{EisenachMirror2} and~\ref{EisenachMirror3}  in
the next example for a special case with $q_0$ chosen as
a so-called finite-zone potential near $\pm\infty$; cf.
\cite{Levitan87,Levitan90}.

\begin{example}
Let $n_\pm\in\mathbb N$ and consider finitely many real numbers
$\lambda_k^\pm$, $\lambda^+_{n_++1}$, $\mu_k^\pm$, $\mu^-_{n_-+1}$,
$k=1,\dots,n_\pm$, such that
\begin{equation*}
\begin{split}
  &\lambda_1^+<\mu_1^+<\lambda_2^+<\mu_2^+<\dots
<\lambda^+_{n_+}<\mu^+_{n_+}<\lambda^+_{n_++1},\\
  &\mu^-_{n_-+1}<\lambda^-_{n_-}<\mu^-_{n_-}<\dots
<\lambda_2^-<\mu_2^-<\lambda_1^-<\mu_1^-.
  \end{split}
  \end{equation*}
  Let $(a,b)=\mathbb R$ and
for $r_0$ in Hypothesis~\ref{h1} we shall assume, in addition, that
$r_0=1$ on $(\beta,\infty)$ and $r_0=-1$ on $(-\infty,\alpha)$.
Moreover, we set $p_0=1$ on $(-\infty,\alpha)$ and $(\beta,\infty)$. In
other words,
the differential expression $\ell_0$ in \eqref{ell} leads to operators
$\pm H_{0,\pm}$ of the form
\begin{equation*}
H_{0,+}=-\frac{\mathrm d^2}{\mathrm dx^2}+q_{0,+}\quad\text{and}\quad
  -H_{0,-}=\frac{\mathrm d^2}{\mathrm dx^2}-q_{0,-},
\end{equation*}
with Dirichlet boundary conditions at $\beta$ 
(for $H_{0,+}$) and at $\alpha$ (for $-H_{0,-}$), and
where $q_{0,\pm}$ denote the restrictions of $q_0$ onto $(\beta,\infty)$
and $(-\infty,\alpha)$, respectively.
According to \cite[Chapter 8]{Levitan87} there exist finite-zone
potentials $q_{0,\pm}$ such that Hypothesis~\ref{h2} is satisfied,
\begin{equation}\label{jabitteschoen}
\begin{split}
   \sigma_\mathrm{ess}(H_{0,+})&=[\lambda_1^+,\mu_1^+]\cup
[\lambda_2^+,\mu_2^+]\cup \dots \cup
[\lambda^+_{n_+},\mu^+_{n_+}]\cup[\lambda^+_{n_++1},\infty),\\
\sigma_\mathrm{ess}(-H_{0,-})&=(-\infty,\mu^-_{n_-+1}]\cup[\lambda^-_{n_-},\mu^-_{n_-}]\cup
\dots \cup [\lambda_2^-,\mu_2^-]\cup[\lambda_1^-,\mu_1^-],
\end{split}
\end{equation}
and $\pm H_{0,\pm}$ have at most finitely many discrete (real) eigenvalues. 

According to Theorem~\ref{EisenachMirror2} the nonreal spectrum of $K_0$ 
consists of discrete eigenvalues (with geometric multiplicity one) which are contained
in a compact subset of $\mathbb C$. Possible accumulation points of nonreal eigenvalues 
are contained in the intersection of the two sets of bands of essential spectra in \eqref{jabitteschoen}.
Note that in the present situation possible points $\mu_i^+=\lambda_j^-$ for some $i,j$ or 
$\mu_l^-=\lambda_k^+$ for some $l,k$ satisfy property $(${\rm P}$)$ in Theorem~\ref{EisenachMirror2}~(ii) as 
$\pm H_{0,\pm}$ have at most finitely many discrete eigenvalues. 
Hence, if $\mu^-_{n_-+1}\leq \lambda_1^+$ and $\mu_1^-\leq \lambda^+_{n_++1}$ and
\begin{equation}\label{nichtwitzig}
(\lambda_k^+,\mu_k^+)\cap(\lambda_l^-,\mu_l^-) =\emptyset
\end{equation}
for all
$k,l$, then $K_0$ has at most finitely many nonreal eigenvalues.

Concerning the real discrete eigenvalues Theorem~\ref{EisenachMirror3} implies the following: real discrete
eigenvalues of $K_0$ may only accumulate  (from the left) to points $\lambda_k^+=\lambda_l^-$
or (from the right) to points $\mu_i^+=\mu_j^-$ for some $k,l$ or $i,j$. Therefore,
under the assumptions
\begin{equation}\label{nichtwitzig2}
\lambda_k^+\not=\lambda_l^-\quad\text{and}\quad\mu_i^+\not=\mu_j^- 
\end{equation}
for all $k,l$ 
and $i,j$ the operator $K_0$
has at most finitely many discrete real eigenvalues. Note that condition \eqref{nichtwitzig} implies \eqref{nichtwitzig2}
and hence it follows under the assumptions $\mu^-_{n_-+1}\leq \lambda_1^+$, $\mu_1^-\leq \lambda^+_{n_++1}$, and \eqref{nichtwitzig}
that  $K_0$ has at most finitely many discrete (real and
nonreal) eigenvalues.
\end{example}
\medskip

In the next theorem we discuss a more special situation that arises 
when, roughly speaking, the essential spectra of $H_{0,+}$ and 
$-H_{0,-}$  are separated by a gap. More precisely, if
 $\eta_{\pm}$ denote 
the lower bounds of the essential spectra of $H_{0,\pm}$,
$$
\eta_{+} := \min \sigma_\mathrm{ess}(H_{0,+}) \quad \mbox{and} \quad 
\eta_{-} := \min \sigma_\mathrm{ess}(H_{0,-}),
$$
where 
$\eta_{\pm} :=\infty$ if $\sigma_\mathrm{ess}(H_{0,\pm})=\emptyset$, then  we shall assume that 
$$-\eta_- <\eta_+.$$
In the following we fix some 
\begin{equation}\label{etachen}
\eta \in (-\eta_{-},\eta_{+}).
\end{equation}
Let $H_0$ be the operator in \eqref{ortho} and  define the 
hermitian form $\langle\cdot,\cdot\rangle$ on $\operatorname{dom} H_0$ for $f,g \in \operatorname{dom} H_0$ by
\begin{equation*}
\begin{split}
\langle f,g\rangle&:=[(H_0-\eta)f,g] =
\left( (-H_{0,-}-\eta)f_-,f_-\right)_{L^2((a,\alpha); r_0)}\\[1ex]
&+ \left( (K_{0,\alpha\beta}-\eta)f_{\alpha\beta},f_{\alpha\beta}\right)_{L^2((\alpha,\beta);r_0)}
+ \left( (H_{0,+}-\eta)f_+,f_+\right)_{L^2((\beta,b);r_0)},
\end{split}
\end{equation*}
where $f,g$ have the obvious decomposition with respect to~\eqref{orthokkk} with $j=0$, 
$f=f_-+f_{\alpha\beta}+f_+$ and $g=g_-+g_{\alpha\beta}+g_+$.
Then the operator $H_0-\eta$ has finitely many negative squares and the number
$\kappa$ of negative squares is given by the 
sum of the negative squares of the three entries on the diagonal, that is,
\begin{equation}\label{Haka}
\kappa =\kappa_{+} + \kappa_{\eta} + \kappa_{-},
\end{equation}
where 
\begin{eqnarray*}
&&\kappa_{\eta} \quad \mbox{is the number of negative squares of the operator } K_{0,\alpha\beta}-\eta,\\
&& \kappa_{+}  \quad \mbox{is the number 
of eigenvalues of the operator } -H_{0,-}-\eta \mbox{ in } (0,\infty),\\ 
&& \kappa_{-} \quad \mbox{is the number 
of eigenvalues of the operator } H_{0,+}-\eta \mbox{ in } (-\infty,0).
\end{eqnarray*}
This is due to the fact that the weight $r_0$ is negative on 
$(a,\alpha)$ and positive on $(\beta,b)$.
In the next theorem we discuss further properties of the eigenvalues of $K_0$ and their algebraic eigenspaces.

\begin{theorem}\label{EisenachAuckland}
Assume Hypotheses~\ref{h1}, \ref{h2}, and \ref{h3}, and let $\eta_\pm,\eta$ be as above. 
Then the operator $K_0-\eta$ has at most 
$$
\kappa_0 := \kappa_{+} + \kappa_{\eta} + \kappa_{-} +2
$$
negative squares
which implies the following statements for the eigenvalues of $K_0$.
\begin{enumerate}
\item [{\rm (i)}] There are at most $\kappa_0$ different 
real eigenvalues of $K_0$ with 
corresponding Jordan chains of length greater than one. The length of each of these 
chains is at most $2\kappa_0 +1$.
\item[{\rm (ii)}] The nonreal spectrum of $K_0$ consists of at most $\kappa_0$ pairs $\{\mu_{i},\overline\mu_{i}\}$,
$\operatorname{Im}\mu_{i}>0$, of discrete eigenvalues with corresponding Jordan chains
of length at most $\kappa_0$.
\end{enumerate}  
\end{theorem}

\begin{proof}
The operator $H_0-\eta$ has $\kappa$ squares with $\kappa$ as in \eqref{Haka}. 
By~\eqref{pertu}
the resolvent difference 
of the resolvents of 
the operators $H_0$ and $K_0$ is a rank two operator and, hence, by
Theorem~\ref{Reklame} (v) the operator $K_0-\eta$ has at most $\kappa +2$
negatives  squares. Now the assertions in Theorem~\ref{EisenachAuckland}
follow from  Theorem~\ref{Reklame}.
\end{proof}

Next we provide quantitative estimates on the number of eigenvalues of 
$K_0$ in the setting of Theorem~\ref{EisenachAuckland}. Recall that a finite rank perturbation of self-adjoint operators in 
Krein spaces may change the discrete spectrum dramatically, in particular, for general self-adjoint operators in (infinite dimensional) 
Krein spaces the number of eigenvalues in a gap of the essential spectrum may change arbitrarily under a rank one or rank two perturbation; cf. 
\cite[Theorem 3.1]{BP16}. The situation is different if the operators have finitely many negative squares;
in the next example we illustrate how the estimates for rank perturbations in \cite{BMT16} can be applied successively in the present situation to 
obtain upper bounds on the number of eigenvalues.

\begin{example}
Consider the case $\eta=0$ in \eqref{etachen} and in Theorem~\ref{EisenachAuckland}, and fix an interval 
$I\subset \mathbb R\setminus \sigma_\mathrm{ess}(H_0)$ such that $I\subset (0,\eta_+)$
or $I\subset (-\eta_-,0)$. 
In the following we denote the number of eigenvalues of a closed operator $A$ in $I$ by $n_A(I)$. 

In the case $\alpha<\beta$ we make use of the auxiliary self-adjoint operator $K_{0,\alpha}$ in the Krein
space $L^2((a,b);r_0)$,
which is defined as the direct sum of the maximal realizations of $\ell_0$ in $L^2((a,\alpha);r_0)$ and $L^2((\alpha, b);r_0)$
with Dirichlet boundary conditions
at $\alpha$. Then $\rho(K_{0,\alpha})\not=\emptyset$ follows in the same way as in the proof of Theorem~\ref{Eisenach} 
from \cite[Theorem 4.5]{BP10} and since the resolvents of $K_{0,\alpha}$ and $H_0$ differ by a rank one operator we can apply \cite[Corollary~3.2]{BMT16}
to $I$ and obtain
\begin{equation}\label{asas}
n_{K_{0,\alpha}}(I)\leq  n_{H_0}(I)+ n_{H_0,K_{0,\alpha}}(I) +2\kappa +3,
\end{equation}
where $n_{H_0,K_{0,\alpha}}(I)$ stands for the number of joint eigenvalues
of the operators $H_0$ and $K_{0,\alpha}$ in $I$ and
 $\kappa$ is the number of negative squares of $H_0$ in 
\eqref{Haka}. Of course, $n_{H_0,K_{0,\alpha}}(I)\leq n_{H_0}(I)$
which gives
\begin{equation} \label{Haka7}
n_{K_{0,\alpha}}(I)\leq  2n_{H_0}(I)+ 2\kappa +3.
\end{equation}
The same argument for the operators $K_0$
and $K_{0,\alpha}$, where one also uses that the operator 
$K_{0,\alpha}$ has at most $\kappa +1$ negative squares by Theorem~\ref{Reklame}~(v), leads to the estimate
$$
n_{K_{0}}(I)\leq 2n_{K_{0,\alpha}}(I)+ 2(\kappa+1) +3,
$$
and with \eqref{Haka7} we conclude
\begin{equation}\label{Tatuara}
n_{K_{0}}(I)\leq 4n_{H_0}(I)+ 4\kappa +6 + 2(\kappa+1) +3 =
4n_{H_0}(I) + 6\kappa + 11.
\end{equation}
Under the additional assumption $I\subset\rho(H_0)$ the estimate \eqref{Tatuara} improves and reduces to $
n_{K_{0}}(I)\leq  6\kappa + 11$.

In the case $\alpha=\beta$ the above reasoning simplifies (see also Remark~\ref{remarkspecial}) and one 
can directly apply \cite[Corollary~3.2]{BMT16} to the operators $K_0$ and $H_0$, so that instead of \eqref{asas} we obtain immediately
\begin{equation}\label{asasas}
n_{K_{0}}(I)\leq  n_{H_0}(I)+ n_{H_0,K_{0}}(I) +2\kappa +3,
\end{equation}
where $\kappa$ in \eqref{Haka} now has the form $\kappa =\kappa_{+} + \kappa_{-}$.
Note that under the additional assumption $I\subset\rho(H_0)$ the estimate \eqref{asasas} simplifies to $
n_{K_{0}}(I)\leq  2\kappa + 3$.
\end{example}
\medskip

Next we take a closer look at the discrete eigenvalues of $K_0$ in the
situation of Corollary~\ref{Erfurt} and consider again the case
$(a,b)=\mathbb R$, where the coefficients $r_0,p_0,q_0$ admit the limits
        \begin{equation}\label{dasgesundeplusNullabc}
                r_{\pm\infty}=\lim_{x\rightarrow \pm\infty} r_0(x),\,\, p_{\pm\infty}=
\lim_{x\rightarrow \pm\infty} p_0(x),\quad\text{and}\quad q_{\pm\infty}=
\lim_{x\rightarrow \pm\infty} q_0(x),
        \end{equation}
with $\pm r_{\pm\infty}>0$, $p_{\pm\infty}>0$, $q_{\pm\infty}\in\mathbb R$,
        and, in addition, we assume
        \begin{equation}\label{gapda}
        -q_{-\infty}/r_{-\infty}<q_{\infty}/r_{\infty}.
        \end{equation}
        Then, by Corollary~\ref{Erfurt}, the essential spectrum of $K_0$ is given by
        \begin{equation}
                \label{essspec}
                \sigma_{\mathrm{ess}}(K_0) =
\left(-\infty,-\frac{q_{-\infty}}{r_{-\infty}}\right] \cup
\left[\frac{q_{\infty}}{r_{\infty}},\infty\right)
                \end{equation}
                and by \eqref{gapda} there is a gap in \eqref{essspec}, that is, the set
                $\sigma(K_0)\cap (-q_{-\infty}/r_{-\infty}, q_{\infty}/r_{\infty})$
consists of discrete eigenvalues. It is also clear from Theorem~\ref{EisenachAuckland} that 
$K_0-\eta$
                with $-q_{-\infty}/r_{-\infty}<\eta< q_{\infty}/r_{\infty}$
 has 
finitely many negative squares, and hence the nonreal spectrum of $K_0$ consists of at most finitely many
discrete eigenvalues.
               
We proceed by providing a Kneser type
result for the indefinite Sturm--Liouville operator $K_0$, i.e., we
obtain criteria for the
accumulation/non-accumulation of
eigenvalues of $K_0$ in $(-q_{-\infty}/r_{-\infty},
q_{\infty}/r_{\infty})$ to the edges of the essential spectrum
(see, e.g., \cite[Theorem~9.42 and Corollary~9.43]{te} or \cite{BSTT23,kt3} for such type of 
results in the definite case).
We first recall some notation from \cite{BSTT23,kt3}: the iterated logarithm $\log_n$ is defined
recursively via
\[
\log_0(x) := x, \qquad \log_n(x) := \log(\log_{n-1}(x)),\quad n\in\mathbb N,
\]
with the convention $\log(x):=\log|x|$ for negative values of
$x$. Then
$\log_n(x)$ will be continuous for $x>\mathrm{e}_{n-1}$ and positive for
$x>\mathrm{e}_n$, where
$\mathrm{e}_{-1} :=-\infty$ and $\mathrm{e}_n
:=\mathrm{e}^{\mathrm{e}_{n-1}}$. Furthermore, abbreviate 
\[
L_n(x) := \frac{1}{\log_{n+1}'(x)} = \prod_{j=0}^n \log_j(x)
\]
and
\begin{equation*}
  P_n(x):=\sum_{j=0}^{n-1}  \frac{1}{L_j(x)} \quad\text{and}\quad
Q_n(x) := -\frac{1}{4} \sum_{j=0}^{n-1} \frac{1}{L_j(x)^2}.
\end{equation*}
Here the convention $\sum_{j=0}^{-1} \equiv 0$ is used, so that, $P_0(x)=Q_0(x)=0$.

\begin{theorem}\label{Kneser}
Let $(a,b)=\mathbb R$ and assume that the coefficients $r_0,p_0,q_0$
admit the limits in \eqref{dasgesundeplusNullabc} with $\pm
r_{\pm\infty}>0$, $p_{\pm\infty}>0$, $q_{\pm\infty}\in\mathbb R$ such that
\eqref{gapda} holds.
For $n\in\mathbb N_0$ and $x\in \mathbb R\setminus
[-\mathrm{e}_n,\mathrm{e}_n]$ let
        \begin{equation*} 
\Delta_{0,\pm}(x):=  L_n(x)^2 \left(\frac{q_0(x)}{p_{\pm\infty}}-
Q_n(x) -\frac{q_{\pm\infty}}{p_{\pm\infty} r_{\pm\infty}}r_0(x)
   + \frac{P_n(x)^2}{4}
\Big(1-\frac{p_{\pm\infty}}{p_0(x)}\Big) \right).
        \end{equation*}
        Then the set $\sigma(K_0)\cap (-q_{-\infty}/r_{-\infty},
q_\infty/r_\infty)$ consists of discrete eigenvalues of $K_0$ which
accumulate
        at $\pm q_{\pm \infty}/r_{\pm \infty}$ if
\begin{equation*}
                              \limsup_{x\rightarrow \pm\infty}{}
                                                \Delta_{0,\pm} (x) < -\frac{1}{4}
                      \end{equation*}
                     and do not accumulate at $\pm q_{\pm \infty}/r_{\pm \infty}$ if
                      \begin{equation*}
                              \liminf_{x\rightarrow \pm\infty}{}
                                                \Delta_{0,\pm} (x) > -\frac{1}{4}.
                      \end{equation*}
\end{theorem}

\begin{proof}
It follows from \cite[Theorem 3.5]{BSTT23} that
$\sigma_\mathrm{ess}(H_{0,+})=[q_\infty/r_\infty,\infty)$ and the set
$\sigma(H_{0,+})\cap (-\infty, q_\infty/r_\infty)$ consists of discrete
eigenvalues of $H_{0,+}$ which accumulate
        at $q_{\infty}/r_{\infty}$ if
\begin{equation*}
                              \limsup_{x\rightarrow\infty}{}
                                                \Delta_{0,+} (x) < -\frac{1}{4}
                      \end{equation*}
                     and do not accumulate at $q_{\infty}/r_{\infty}$ if
                      \begin{equation*}
                              \liminf_{x\rightarrow \infty}{}
                                                \Delta_{0,+} (x) > -\frac{1}{4}.
                      \end{equation*}
In the same way one obtains that
$\sigma_\mathrm{ess}(-H_{0,-})=(\infty,-q_{-\infty}/r_{-\infty})$ and
the set
$\sigma(-H_{0,-})\cap (-q_{-\infty}/r_{-\infty},\infty)$ consists of
discrete eigenvalues of $-H_{0,-}$ which accumulate
        at $-q_{-\infty}/r_{-\infty}$ if
\begin{equation*}
                              \limsup_{x\rightarrow-\infty}{}
                                                \Delta_{0,-} (x) < -\frac{1}{4}
                      \end{equation*}
                     and do not accumulate at $q_{-\infty}/r_{-\infty}$ if
                      \begin{equation*}
                              \liminf_{x\rightarrow -\infty}{}
                                                \Delta_{0,-} (x) > -\frac{1}{4}.
                      \end{equation*}
The assumption \eqref{gapda} ensures that $q_\infty/r_\infty\not\in
\sigma_\mathrm{ess}(-H_{0,-})$ and
$-q_{-\infty}/r_{-\infty}\not\in \sigma_\mathrm{ess}(H_{0,+})$ and hence
the assertions follow directly from Theorem~\ref{EisenachMirror3}.
\end{proof}

\begin{remark}
As mentioned above the operator $K_0-\eta$
                with $-q_{-\infty}/r_{-\infty}<\eta< q_{\infty}/r_{\infty}$
 has 
finitely many negative squares.
In the case that the discrete eigenvalues of $K_0$ in $(-q_{-\infty}/r_{-\infty},
q_\infty/r_\infty)$ do not accumulate to $\pm q_{\pm \infty}/r_{\pm \infty}$ 
one may also choose $\eta=\pm q_{\pm \infty}/r_{\pm \infty}$, that is, 
also the operator 
 $K_0\pm q_{\pm \infty}/r_{\pm \infty}$
 has 
finitely many negative squares; cf. Theorem~\ref{EisenachAuckland}.
\end{remark}

For the case $n=0$ Theorem~\ref{Kneser}
reduces to
the following statement.

\begin{corollary}\label{DonOmar2}
Let $(a,b)=\mathbb R$ and assume that the coefficients $r_0,p_0,q_0$
admit the limits in \eqref{dasgesundeplusNullabc} with $\pm
r_{\pm\infty}>0$, $p_{\pm\infty}>0$, $q_{\pm\infty}\in\mathbb R$ such that
\eqref{gapda} holds.
        Then the set $\sigma(K_0)\cap (-q_{-\infty}/r_{-\infty},
q_\infty/r_\infty)$ consists of discrete eigenvalues of $K_0$ which
accumulate
        at $\pm q_{\pm \infty}/r_{\pm \infty}$ if
        \begin{equation*}
\limsup_{x\rightarrow\pm\infty}x^2\left(q_0(x)-
\frac{q_{\pm\infty}}{r_{\pm\infty}}r_0(x)\right) < -\frac{p_{\pm\infty}}{4}
                      \end{equation*}
                     and do not accumulate at $\pm q_{\pm \infty}/r_{\pm \infty}$ if
                      \begin{equation*}
\liminf_{x\rightarrow\pm\infty}x^2\left(q_0(x)-
\frac{q_{\pm\infty}}{r_{\pm\infty}}r_0(x)\right) > -\frac{p_{\pm\infty}}{4}.
                      \end{equation*}
\end{corollary}

We end this section with the study of
the point spectrum of the perturbed indefinite Sturm-Liouville operator
$K_1$. Under the assumptions of Theorem~\ref{Eisenach}
the indefinite Sturm-Liouville operator $K_1$ is self-adjoint in the
Krein space $L^2((a,b);r_1)$ with $\rho(K_1)\neq \emptyset$ and
the essential spectrum of $K_1$ is given by
\begin{equation*}
\sigma_\mathrm{ess}(K_0)=
        \sigma_\mathrm{ess}(K_1)=
        \sigma_\mathrm{ess}(H_{1,+}) \cup \sigma_\mathrm{ess}(-H_{1,-}).
\end{equation*}
It is clear that the above considerations and results in this section
also hold for $K_1$. In particular, the Kneser type results
Theorem~\ref{Kneser} and Corollary~\ref{DonOmar2} can be formulated in
the same way with the coefficients $r_1,p_1,q_1$.
However, it is also possible
to use information of the coefficients $r_0,p_0,q_0$ of the unperturbed
operator $K_{0}$
to investigate possible accumulation of discrete eigenvalues of the
perturbed operator $K_1$.

\begin{theorem}\label{Kneser17}
Let $(a,b)=\mathbb R$ and assume that the
coefficients $r_0,p_0,q_0$ admit the limits
\eqref{dasgesundeplusNullabc} 
with $\pm
r_{\pm\infty}>0$, $p_{\pm\infty}>0$, $q_{\pm\infty}\in\mathbb R$ such that
\eqref{gapda} holds. Let
$\Delta_{0,\pm}$ be as in Theorem~\ref{Kneser} and
        assume that
\begin{equation}\label{Anitta}
   \lim_{x\rightarrow \pm \infty}{} L_n(x)^2 \bigg(
        \lvert r_1(x)-r_0(x)\rvert +
        P_n(x)^2
        \left\lvert\frac{1}{p_1(x)}-\frac{1}{p_0(x)}\right\rvert
        + \lvert
        q_1(x)-q_0(x)\rvert \bigg)=0
\end{equation}
holds for some $n\in\mathbb N$, $n\geq 1$. Then the following assertions
hold.
\begin{itemize}
  \item [{\rm (i)}] The essential spectrum of $K_1$ is given by
  \begin{equation*}
                \sigma_{\mathrm{ess}}(K_1) =
\left(-\infty,-\frac{q_{-\infty}}{r_{-\infty}}\right] \cup
\left[\frac{q_{\infty}}{r_{\infty}},\infty\right),
                \end{equation*}
                the nonreal spectrum of $K_1$ consists of at most finitely many
discrete eigenvalues, and
                $K_1-\eta$
                with $-q_{- \infty}/r_{- \infty}<\eta< q_{ \infty}/r_{ \infty}$ has finitely many negative squares.
  \item [{\rm (ii)}] The set $\sigma(K_1)\cap (-q_{-\infty}/r_{-\infty},
q_\infty/r_\infty)$ consists of discrete eigenvalues of $K_1$ which
accumulate
        at $\pm q_{\pm \infty}/r_{\pm \infty}$ if
\begin{equation*}
                              \limsup_{x\rightarrow \pm\infty}{}
                                                \Delta_{0,\pm} (x) < -\frac{1}{4}
                      \end{equation*}
                     and do not accumulate at $\pm q_{\pm \infty}/r_{\pm \infty}$ if
                      \begin{equation*}
                              \liminf_{x\rightarrow \infty}{}
                                                \Delta_{0,\pm} (x) > -\frac{1}{4}.
                      \end{equation*}
\end{itemize}
\end{theorem}
\begin{proof}
Observe first that
                \eqref{Anitta} implies
                \begin{equation}\label{gutgut}
        \begin{split}
        &\lim_{x\rightarrow \infty}{} L_n(x)^2
        \lvert r_1(x)-r_0(x)\rvert =0,\\[1ex]
        &\lim_{x\rightarrow  \infty}{}  L_n(x)^2P_n(x)^2
        \left\lvert\frac{1}{p_1(x)}-\frac{1}{p_0(x)}\right\rvert=0,\\[1ex]
        &\lim_{x\rightarrow  \infty}{}  L_n(x)^2 \lvert q_1(x)-q_0(x)\rvert =0,
        \end{split}
        \end{equation}
and since $\lim_{x\rightarrow \infty} L_n(x)=\infty$,
$$
L_n(x)^2P_n(x)^2=\bigg(\sum_{j=0}^{n-1} \frac{L_n(x)}{L_j(x)}\bigg)^2,
$$
and $L_n(x) \geq L_j(x)$ for $\vert x \vert > e_n$, $j=1, \ldots, n-1$,
it follows from
        \eqref{dasgesundeplusNullabc} and \eqref{gutgut} that
        \begin{equation}\label{qqq}
                r_{\pm\infty}=\lim_{x\rightarrow \pm\infty} r_1(x),\,\, p_{\pm\infty}=
\lim_{x\rightarrow \pm\infty} p_1(x),\quad\text{and}\quad q_{\pm\infty}=
\lim_{x\rightarrow \pm\infty} q_1(x).
        \end{equation}
        Now the statement on the essential spectrum in (i)
of the indefinite Sturm-Liouville operator $K_1$ follows from
Corollary \ref{Erfurt}. It follows from \eqref{qqq}, \eqref{dasgesundeplusNullabc}, and \cite[Theorem 3.2]{BSTT23}
that Hypotheses~\ref{h1}, \ref{h2}, and \ref{h3} hold for the differential expression 
$\tau_1$ with coefficients $r_1,p_1,q_1$.
Hence Theorem~\ref{EisenachMirror2} 
implies that the nonreal spectrum of $K_1$ consists of
at most finitely many discrete eigenvalues and from
Theorem~\ref{EisenachAuckland} applied to $K_1$ we obtain that the
operator $K_1-\eta$
                with $-q_{- \infty}/r_{- \infty}<\eta< q_{ \infty}/r_{ \infty}$ has finitely many negative squares.
       
We now turn to the proof of the accumulation properties of the discrete
eigenvalues of $K_1$ in (ii). For this,
define (in the same way as $\Delta_{0,\pm}$) the functions
\begin{align*}
\Delta_{1,\pm}(x):= & L_n(x)^2 \Bigg(\frac{q_1(x)}{p_{\pm\infty}}-
Q_n(x) -\frac{q_{\pm\infty}}{p_{\pm\infty} r_{\pm\infty}}r_1(x)  +
\frac{P_n(x)^2}{4}
\bigg(1-\frac{p_{\pm\infty}}{p_1(x)}\bigg) \Bigg)
        \end{align*}
for $\vert x\vert >\mathrm{e}_n$. It is easy to see that
        \begin{equation*}
        \begin{split}
        \Delta_{1,\pm}(x)-\Delta_{0,\pm}(x)&=L_n(x)^2
\Bigg(\frac{q_1(x)-q_0(x)}{p_{\pm\infty}}-
        \frac{q_{\pm\infty}}{p_{\pm\infty} r_{\pm\infty}}(r_1(x)-r_0(x))\\
        &\qquad\qquad\qquad\qquad\quad + \frac{p_{\pm\infty}P_n(x)^2}{4}
\bigg(-\frac{1}{p_1(x)}+\frac{1}{p_0(x)}\bigg)
        \Bigg),
        \end{split}
        \end{equation*}
        and from \eqref{gutgut} we conclude
        \begin{equation*}
         \lim_{x\rightarrow
\pm\infty}\bigl(\Delta_{1,\pm}(x)-\Delta_{0,\pm}(x)\bigr)=0.
        \end{equation*}
        Therefore, the limes superior and limes inferior of
        $\Delta_{1,\pm}$ and $\Delta_{0,\pm}$ at $\pm\infty$ coincide.
        Now assertion (ii) follows from Theorem~\ref{Kneser}
        applied to $K_1$.
\end{proof}

Next we consider the case $n=0$ which was excluded in the previous result,
as one has to assume in this case, in addition, that $p_1$ and $p_0$
have the same
limits at $\pm \infty$.

\begin{corollary}
Let $(a,b)=\mathbb R$ and assume that the
coefficients $r_0,p_0,q_0$ admit the limits
\eqref{dasgesundeplusNullabc} with $\pm
r_{\pm\infty}>0$, $p_{\pm\infty}>0$, $q_{\pm\infty}\in\mathbb R$ such that
\eqref{gapda} holds. If
        \begin{equation*}
   \lim_{x\rightarrow \pm\infty}{} x^2 \bigl(
        \lvert r_1(x)-r_0(x)\rvert +
        \lvert
        q_1(x)-q_0(x)\rvert \bigr)=0 \quad
        \mbox{and }\lim_{x\rightarrow \pm\infty}{}
        \lvert p_1(x)-p_0(x)\rvert=0,
\end{equation*}
        then (i) in Theorem~\ref{Kneser17} holds and the set $\sigma(K_1)\cap
(-q_{-\infty}/r_{-\infty}, q_\infty/r_\infty)$ consists of discrete
eigenvalues of $K_1$ which accumulate
        at $\pm q_{\pm \infty}/r_{\pm \infty}$ if
\begin{equation*}
\limsup_{x\rightarrow\pm\infty}x^2\left(q_0(x)-
\frac{q_{\pm\infty}}{r_{\pm\infty}}r_0(x)\right) < -\frac{p_{\pm\infty}}{4}
                      \end{equation*}
                     and do not accumulate at $\pm q_{\pm \infty}/r_{\pm \infty}$ if
                      \begin{equation*}
\liminf_{x\rightarrow\pm\infty}x^2\left(q_0(x)-
\frac{q_{\pm\infty}}{r_{\pm\infty}}r_0(x)\right) > -\frac{p_{\pm\infty}}{4}.
                      \end{equation*}
\end{corollary}

\section{Periodic problems}\label{section4}

In this section we study indefinite Sturm-Liouville operators with periodic coefficients near the singular endpoints $\pm\infty$
and we treat $L^1$-perturbations in the spirit of \cite{BSTT24}, see also \cite{BMT11,DL86,K11,P13} for related considerations in the indefinite setting.
More precisely, we shall assume that Hypothesis~\ref{h1} holds and that
the coefficients
$1/p_0,q_0,r_0$ are $\omega$-periodic in $(\beta,\infty)$ and
$\theta$-periodic in $(-\infty,\alpha)$ for some
$\omega,\theta>0$; this also implies that $\infty$ and $-\infty$ are in
the limit point case, that is, Hypothesis~\ref{h2} is automatically
satisfied. In this situation the essential spectra of
$H_{0,+}$ and $H_{0,-}$ are purely absolutely continuous and consist of
infinitely many closed intervals
\begin{equation}\label{harrypm}
\sigma_{\mathrm{ess}}(H_{0,+}) = \bigcup_{k=1}^\infty
[\lambda_{k}^+,\mu_{k}^+]\quad\text{and}\quad
\sigma_{\mathrm{ess}}(-H_{0,-}) = \bigcup_{l=1}^\infty
[\lambda_l^-,\mu_{l}^-]
\end{equation}
where the endpoints $\lambda_k^+$ and $\mu_{k}^+$,
$\lambda_k^+<\mu_k^+$, denote the $k$-th eigenvalues of the
regular Sturm--Liouville operator in $L^2((\beta,\beta+\omega);r_0)$
(in nondecreasing order) with
periodic and semiperiodic boundary conditions, respectively, and
$-\mu_{l}^-$ and $-\lambda_{l}^-$,
$-\mu_{l}^-<-\lambda_{l}^-$, denote the $l$-th eigenvalues of the
regular Sturm--Liouville operator in $L^2((\alpha-\theta,\alpha);r_0)$ (in
nondecreasing order) with
periodic and semiperiodic boundary conditions, respectively; cf.\
\cite{BrownEasthamSchmidt13} or \cite[Section 12]{Weidmann87}.

Below we determine the essential spectra of indefinite Sturm operators
$K_0$ and $K_1$, where $K_0$ is periodic near the endpoints and
$K_1$ is an $L^1$-perturbation.

\begin{theorem}\label{Lad}
Assume that Hypothesis~\ref{h1} holds and
suppose that $1/p_0,q_0,r_0$ are $\omega$-periodic in $(\beta,\infty)$ and
$\theta$-periodic in $(-\infty,\alpha)$ for some $\omega,\theta>0$. Then the
following assertions hold.
\begin{itemize}
  \item [{\rm (i)}] The indefinite Sturm-Liouville operator $K_0$ is
self-adjoint in the Krein space $L^2(\mathbb R;r_0)$, the
resolvent set $\rho(K_0)$ is nonempty, and the essential spectrum is
given by
\begin{equation}
                \label{StrongMint334}
                \sigma_\mathrm{ess}(K_{0})=
                \bigcup_{k=1}^\infty [\lambda_{k}^+,\mu_{k}^+] \cup
                \bigcup_{l=1}^\infty [\lambda_l^-,\mu_{l}^-].
        \end{equation}
\item [{\rm (ii)}]
If
\begin{align}\label{L1nearEnd}
\int_{\mathbb R} \left(\lvert r_1(t)-r_0(t)\rvert +
\left\lvert\frac{1}{p_1(t)}-\frac{1}{p_0(t)}\right\rvert + \lvert
q_1(t)-q_0(t)\rvert \right)
\,\mathrm dt <\infty,
\end{align}
then also the indefinite Sturm-Liouville operator $K_1$ is self-adjoint
in the Krein space $L^2(\mathbb R;r_1)$, the
resolvent set $\rho(K_1)$ is nonempty, and the essential spectrum is
given by
        \begin{equation}
                \label{StrongMint2}
                \sigma_{\mathrm{ess}}(K_1) =
                \bigcup_{k=1}^\infty [\lambda_{k}^+,\mu_{k}^+] \cup
                \bigcup_{l=1}^\infty [\lambda_l^-,\mu_{l}^-].
        \end{equation}
        \end{itemize}
\end{theorem}
\begin{proof}
(i) The periodicity of the coefficients $1/p_0,q_0,r_0$ implies the
limit point condition at both singular endpoints $\pm\infty$.
Hence the maximal operator $L_0$ in \eqref{maxis} is self-adjoint in the Hilbert space
$L^2(\mathbb R;\vert r_0\vert)$. Define $H_{0,+}$ and $H_{0,-}$ as in \eqref{teile}. Then $H_{0,+}$ and $H_{0,-}$ are both semibounded from
below; cf.\ \cite{BrownEasthamSchmidt13} or \cite{Weidmann87}. This implies the semiboundedness of $L_0$.
Hence we conclude from \cite[Theorem 4.5]{BP10} that
the indefinite Sturm-Liouville operator $K_0$ is self-adjoint in the
Krein space $L^2(\mathbb R;r_0)$ and that $\rho(K_0)$ is nonempty.
Furthermore, taking into account \eqref{harrypm} it follows that the
essential spectrum of the block diagonal
operator $H_0$ in \eqref{ortho2}
is given by
\begin{equation*}
                \sigma_\mathrm{ess}(H_{0})= \sigma_\mathrm{ess}(H_{0,+})\cup
\sigma_\mathrm{ess}(-H_{0,-})=
                \bigcup_{k=1}^\infty [\lambda_{k}^+,\mu_{k}^+] \cup
                \bigcup_{l=1}^\infty [\lambda_l^-,\mu_{l}^-].
        \end{equation*}
Now the same perturbation argument as in the end of the proof of
Theorem~\ref{Eisenach} leads to \eqref{StrongMint334}.
\\
\noindent (ii) The assumption \eqref{L1nearEnd} together with
\cite[Theorem 2.1]{BSTT24} implies
that the operators $H_{1,-}$ and $H_{1,+}$ are self-adjoint and
semibounded from below in $L^2((-\infty,\alpha);\vert r_1\vert)$
and $L^2((\beta,\infty);\vert r_1\vert)$, respectively. It then follows that
also the maximal operator $L_1$ is self-adjoint and semibounded in the Hilbert space
$L^2(\mathbb R;\vert r_1\vert)$ and we again conclude from \cite[Theorem~4.5]{BP10} that
the indefinite Sturm-Liouville operator $K_1$ is self-adjoint in the
Krein space $L^2(\mathbb R;r_1)$ and that $\rho(K_1)$ is nonempty.
Since
\begin{equation*}
\sigma_\mathrm{ess}(H_{0,+})=\sigma_\mathrm{ess}(H_{1,+})\quad\text{and}\quad
\sigma_\mathrm{ess}(-H_{0,-})=\sigma_\mathrm{ess}(-H_{1,-})
\end{equation*}
by \eqref{L1nearEnd} and \cite[Theorem 2.1]{BSTT24} we obtain
$\sigma_\mathrm{ess}(H_{0})=\sigma_\mathrm{ess}(H_{1})$ and in the
same way as in the end of
the proof of Theorem~\ref{Eisenach} a perturbation argument leads to
$\sigma_\mathrm{ess}(K_{0})=\sigma_\mathrm{ess}(K_{1})$,
which finally shows \eqref{StrongMint2}.
\end{proof}

For the discrete spectrum of the operator $K_0$ we obtain the following corollary as an immediate consequence of Theorem~\ref{EisenachMirror2}, Theorem~\ref{EisenachMirror3} (see also Remark~\ref{rettung}), and 
Theorem~\ref{Lad}~(i). Note that the periodic operators $\pm H_{0,\pm}$ have no discrete (real) eigenvalues and that in the present situation property (P) holds for points that satisfy $\lambda_{k}^+=\mu_{l}^-$ or $\mu_{k}^+=\lambda_l^-$ for some $k,l=1,\dots,\infty$;
these points are automatically in $\mbox{\rm int}(\sigma_\mathrm{ess}(K_0))$.

\begin{corollary}\label{Lad2}
Assume that Hypothesis~\ref{h1} holds and
suppose that $1/p_0,q_0,r_0$ are $\omega$-periodic in $(\beta,\infty)$ and
$\theta$-periodic in $(-\infty,\alpha)$ for some $\omega,\theta>0$. Then the
following assertions hold.
\begin{itemize}
	\item [{\rm (i)}] The nonreal spectrum of $K_0$ consists 
of discrete eigenvalues with geometric multiplicity one which are contained
in a compact subset of $\mathbb C$ and  may only accumulate to points in 
	$$[\lambda_{k}^+,\mu_{k}^+] \cap [\lambda_l^-,\mu_{l}^-],\qquad k,l=1,\dots,\infty.$$
	Furthermore, the nonreal eigenvalues do not accumulate to points that satisfy $\lambda_{k}^+=\mu_{l}^-$ or $\mu_{k}^+=\lambda_l^-$ for some $k,l=1,\dots,\infty$.
	\item [{\rm (ii)}] 
If $\lambda_k^+\not\in [\lambda_l^-,\mu_{l}^-]$ ($\mu_k^+\not\in [\lambda_l^-,\mu_{l}^-]$) for all $l=1,\dots,\infty$, then the real discrete eigenvalues of $K_0$ 
do not accumulate from the left at $\lambda_k^+$ (from the right at $\mu_k^+$, respectively). In particular, if $\mu_1^-<\mu_{i}^+$, then 
$(\mu_k^+,\lambda_{k+1}^+)\cap\sigma_\mathrm{ess}(K_0)=\emptyset$ for all $k\geq i$ and 
there are at most finitely many discrete eigenvalues of $K_0$ in $(\mu_k^+,\lambda_{k+1}^+)$. 
\item [{\rm (iii)}] 
If $\lambda_l^-\not\in [\lambda_k^+,\mu_{k}^+]$ ($\mu_l^-\not\in [\lambda_k^+,\mu_{k}^+]$) for all $k=1,\dots,\infty$, then the real discrete eigenvalues of $K_0$ 
do not accumulate from the left at $\lambda_l^-$ (from the right at $\mu_l^-$, respectively). In particular, if $\lambda_{j}^- < \lambda_1^+$, then 
$(\mu_{l+1}^-,\lambda_{l}^-)\cap\sigma_\mathrm{ess}(K_0)=\emptyset$ for all $l\geq j$ and 
there are at most finitely many discrete eigenvalues of $K_0$ in  $(\mu_{l+1}^-,\lambda_{l}^-)$. 
	\end{itemize}
\end{corollary}

Note that the first statement in Corollary~\ref{Lad2}~(i) holds also for the perturbed operator $K_1$ in Theorem~\ref{Lad}~(ii); this follows
directly from Theorem~\ref{EisenachMirror2}, Remark~\ref{rettung}, and Theorem~\ref{Lad}~(ii). However, in order to exclude accumulation of eigenvalues of $K_1$
one has to impose a stronger ''finite first moment'' condition.

\begin{theorem}
Assume that Hypothesis~\ref{h1} holds and
suppose that $1/p_0,q_0,r_0$ are $\omega$-periodic in $(\beta,\infty)$ and
$\theta$-periodic in $(-\infty,\alpha)$ for some $\omega,\theta>0$.
Assume, in addition, that 
\begin{align}\label{aberja}
\int_{\mathbb R} \left(\lvert r_1(t)-r_0(t)\rvert + \left\lvert\frac{1}{p_1(t)}-\frac{1}{p_0(t)}\right\rvert + \lvert q_1(t)-q_0(t)\rvert \right)
\vert t\vert\,\mathrm dt <\infty.
\end{align}
Then the
following assertions hold.
\begin{itemize}
	\item [{\rm (i)}] The nonreal spectrum of $K_1$ consists 
of discrete eigenvalues with geometric multiplicity one which are contained
in a compact subset of $\mathbb C$ and  may only accumulate to points in 
	$$[\lambda_{k}^+,\mu_{k}^+] \cap [\lambda_l^-,\mu_{l}^-],\qquad k,l=1,\dots,\infty.$$
	Furthermore, the nonreal eigenvalues do not accumulate to points that satisfy $\lambda_{k}^+=\mu_{l}^-$ or $\mu_{k}^+=\lambda_l^-$ for some $k,l=1,\dots,\infty$.
	\item [{\rm (ii)}] 
If $\lambda_k^+\not\in [\lambda_l^-,\mu_{l}^-]$ ($\mu_k^+\not\in [\lambda_l^-,\mu_{l}^-]$) for all $l=1,\dots,\infty$, then the real discrete eigenvalues of $K_1$ 
do not accumulate from the left at $\lambda_k^+$ (from the right at $\mu_k^+$, respectively). In particular, if $\mu_1^-<\mu_{i}^+$, then 
$(\mu_k^+,\lambda_{k+1}^+)\cap\sigma_\mathrm{ess}(K_1)=\emptyset$ for all $k\geq i$ and 
there are at most finitely many discrete eigenvalues of $K_1$ in $(\mu_k^+,\lambda_{k+1}^+)$. 
\item [{\rm (iii)}] 
If $\lambda_l^-\not\in [\lambda_k^+,\mu_{k}^+]$ ($\mu_l^-\not\in [\lambda_k^+,\mu_{k}^+]$) for all $k=1,\dots,\infty$, then the real discrete eigenvalues of $K_1$ 
do not accumulate from the left at $\lambda_l^-$ (from the right at $\mu_l^-$, respectively). In particular, if $\lambda_{j}^- < \lambda_1^+$, then 
$(\mu_{l+1}^-,\lambda_{l}^-)\cap\sigma_\mathrm{ess}(K_1)=\emptyset$ for all $l\geq j$ and 
there are at most finitely many discrete eigenvalues of $K_1$ in  $(\mu_{l+1}^-,\lambda_{l}^-)$. 
	\end{itemize}
\end{theorem}

	\begin{proof}
	Observe first that condition \eqref{aberja} implies \eqref{L1nearEnd} and hence the indefinite Sturm-Liouville operator $K_1$ is self-adjoint
in the Krein space $L^2(\mathbb R;r_1)$, the resolvent set $\rho(K_1)$ is nonempty, and the essential spectrum is given by \eqref{StrongMint2}. 
It follows from the condition \eqref{aberja} and \cite[Theorem 2.3]{BSTT24} that every gap in $\sigma_{\mathrm{ess}}(H_{1,+})$ and every gap in $\sigma_{\mathrm{ess}}(-H_{1,-})$
contains at most finitely many eigenvalues and hence the same is true for the operator $H_1$. Hence property (P) in 
Theorem~\ref{EisenachMirror2}~(ii) holds 
for points that satisfy $\lambda_{k}^+=\mu_{l}^-$ or $\mu_{k}^+=\lambda_l^-$ for some $k,l=1,\dots,\infty$
and now (i) follows from Theorem~\ref{EisenachMirror2}.
Furthermore, the same arguments as in the proof of 
Theorem~\ref{EisenachMirror3} show that the discrete eigenvalues of $K_1$ do not accumulate to boundary points $\lambda_k^+$ or $\mu_k^+$
of the essential spectrum, where $\lambda_k^+\not\in [\lambda_l^-,\mu_{l}^-]$ or $\mu_k^+\not\in [\lambda_l^-,\mu_{l}^-]$ for all $l=1,\dots,\infty$.
This implies (ii); assertion (iii) follows in the same way (see also the discussion below Corollary~\ref{Lad2}).
	\end{proof}

\appendix

\section{Some facts on self-adjoint operators in Krein spaces}\label{appendix}

In this appendix we recall the concept of Krein spaces
and some properties of certain classes of self-adjoint operators therein,
which appear in this paper. We refer the interested reader to the monographs \cite{AI,B,Gheo22} for a thourough introduction to operator theory 
in Krein spaces.

In the following $({\mathcal K},\Skdef)$ denotes a Hilbert space and $J$ is a bounded self-adjoint
operator in $\mathcal K$ with the property $J^2=I$. Define a new inner product by
$$
[x,y]:= (Jx,y), \quad x,y \in \mathcal K.
$$ 
Since $\sigma(J) \subset \{-1,1\}$ we have the \textit{fundamental decomposition} 
$\mathcal K = \mathcal K_+ \oplus \mathcal K_-$, where $\mathcal K_\pm=\ker(J\mp 1)$. Note that $\mathcal K_+$ and $\mathcal K_-$ are 
also orthogonal with respect
to $\Skindef$ and that the inner product $\Skindef$ is \textit{indefinite} if $\mathcal K_\pm\not=\{0\}$, that is, 
 $\Skindef$  takes
positive and negative values: 
$$
[x_+,x_+] >0 \quad \mbox{and} \quad [x_-,x_-] <0,\qquad x_\pm \in \mathcal K_\pm\setminus\{0\}.
$$
The space $({\mathcal K},\Skindef)$  is then called a {\em Krein space} and 
the operator $J$ is called \textit{fundamental symmetry}.
We mention that often Krein spaces are introduced by starting with an indefinite inner product and a fundamental decomposition, 
see \cite{AI,B,Gheo22}.

Let $(\mathcal K,[\cdot,\cdot])$ be a Krein space and let $A$ be a bounded
or unbounded (with respect to the Hilbert space norm) operator in $\mathcal K$. The adjoint $A^+$ is defined 
via the indefinite inner product $[\cdot,\cdot]$ and one has $A^+=JA^*J$, where $^*$ denotes the adjoint with respect to the Hilbert space scalar product
$\Skdef$.
It follows that $A$ is self-adjoint with respect to $\Skdef$ if and only if $JA$ is self-adjoint with respect to $\Skindef$.
It is important to note that the spectral properties of
operators which are self-adjoint with respect to a Krein space inner
product differ essentially from the ones of self-adjoint operators in Hilbert spaces,
e.g., the spectrum is in general not real and may also coincide with $\mathbb C$. However, 
the indefiniteness of the inner product $\Skindef$ can be used to further classify 
eigenvalues of operators in Krein spaces, e.g.
an isolated point $\lambda_0\in\sigma_\mathrm{p}(A)$
is called of {\em positive} ({\em negative})
 {\em type} if all corresponding eigenvectors  $x$ satisfy
$[x,x] > 0$  ($[x,x]<0$, respectively). 
This notion is extended to all points from the approximate point spectrum $\sigma_\mathrm{ap}(A)$ in the next definition. 
Recall that for a self-adjoint operator $A$ in a Krein space all real spectral points belong to $\sigma_\mathrm{ap}(A)$; cf. \cite[Corollary VI.6.2]{B}.

\begin{definition}  \label{definition++}
For a  self-adjoint operator $A$ in the Krein space
$({\mathcal K},\Skindef)$
a point $\lambda_0 \in \sigma(A)$ is called a
spectral point of {\em positive\/} {\rm (}{\em negative\/}{\rm )}
 {\em type of $A$\/}
if $\lambda_0\in \sigma_\mathrm{ap}(A)$ and every sequence
$(x_{n})$ in $\operatorname{dom} A$ with
$\|x_{n} \| =1$ and $\| (A -\lambda_0 )x_{n} \| \to 0$ as
$n \to \infty$ satisfies
\begin{displaymath}
\liminf_{n \to \infty}\, [x_{n},x_{n}] >0 \;\;\;\;
 \bigl(\limsup_{n \to \infty}\, [x_{n},x_{n}] <0,\,respectively \bigr).
\end{displaymath}
\end{definition}

Spectral theory in Krein spaces is often focusing on sign-type spectrum and further properties in a neighbourhood of such spectral points. 
The following proposition is of this nature; for a proof see \cite{LMM97} and also \cite{AJT,J03}.

\begin{proposition}\label{coll}
Let $A$ be a  self-adjoint operator in the Krein space
$({\mathcal K},\Skindef)$ and
let  $\lambda_0$ be a spectral point of positive type of $A$.
Then $\lambda_0$ is real, the  non-real spectrum of $A$ cannot accumulate to $\lambda_0$, and there exists an open neighbourhood ${\mathcal U}$ 
in $\mathbb C$ of $\lambda_0$ such that the following statements hold.
\begin{itemize}
\item[{\rm (i)}] All spectral points in $\mathcal U\cap\mathbb R$ are of positive type.
\item[{\rm (ii)}]
There exists a number $M>0$ such that
\begin{equation*} 
\|(A -\lambda)^{-1}\| \leq \frac{M}{|\mbox{{\rm Im}}\, \lambda|}\;\;\;
\mbox{for all } \lambda \in {\mathcal U}\setminus \mathbb R.
\end{equation*}
\item[{\rm (iii)}]
 There exists a local spectral function of $A$ of positive type:
To each interval $\delta$ with $\overline{\delta}\subset\mathcal U$
 there exists a self-adjoint projection
$E(\delta)$ which commutes with $A$, the space $(E(\delta)\mathcal K,
\Skindef)$ is a Hilbert space and $\sigma(A\vert E(\delta)\mathcal K)
\subset \sigma(A) \cap \overline{\delta}$.
\end{itemize}
An analogous statement holds for spectral points of negative type.
\end{proposition}

 Roughly speaking,
Theorem~\ref{coll} (iii) states that in a neighbourhood of a spectral point of positive type 
$A$ behaves
(locally) like a self-adjoint operator in a Hilbert space.
In what follows we present two classes of operators
with the property that they have intervals with spectrum of positive or negative type:
operators with finitely many squares and locally definitizable operators.

We shall say that a self-adjoint operator $A$ in a Krein space $(\mathcal K, \Skindef)$
has $\kappa$ {\it negative squares}, $\kappa\in\mathbb N_0$,
if the hermitian form $\langle\cdot,\cdot\rangle$ on $\operatorname{dom} A$, defined by
\begin{equation*}
\langle f,g\rangle:=[Af,g],\qquad 
f,g \in \operatorname{dom}  A,
\end{equation*}
has $\kappa$ negative squares, that is, there exists a
$\kappa$-dimensional subspace $\mathcal M$ in $\operatorname{dom} A$ such that $\langle 
v, v\rangle<0$ if $ v\in\mathcal M$, $ v\not=0$, but no
$\kappa+1$ dimensional subspace with this property. 
In the following theorem we recall some spectral properties of self-adjoint operators
with finitely many negative squares. The statements are well known and are 
consequences of the general results in \cite{CL89,Lhabil,L82}, see also
\cite[Theorem 3.1]{BT07}. In item (iv) below we use the notation $\overline{\mathbb R} := \mathbb R \cup \{ \infty \}$ and
$\overline{\mathbb C} := \mathbb C \cup \{ \infty \}$.
 
\begin{theorem} \label{Reklame}
Let $A$ be a self-adjoint operator in the Krein space $(\mathcal K, \Skindef)$
such that $\rho(A)\not=\emptyset$ and assume that
$A$ has $\kappa$ negative squares. Then the following holds.
\begin{enumerate}
\item[{\rm (i)}] The nonreal spectrum of $A$ consists of at most $\kappa$  pairs $\{\mu_i,\overline\mu_i\}$,
$\operatorname{Im} \mu_i >0$, of discrete eigenvalues.
Denote for an eigenvalue $\lambda$ of $A$ the signature
of the inner product $\Skindef$ on the algebraic eigenspace by
$\{\kappa_-(\lambda), \kappa_0(\lambda),\kappa_+(\lambda)\}$. Then 
\begin{equation}  \label{ReklameSumme}
\sum_{\lambda \in \sigma_\mathrm{p}(A) \cap (-\infty,0)} \!\!\!\!\!\!\!\!\!\!\!\!\!
\left(\kappa_+(\lambda) + \kappa_0(\lambda)\right) \, + \!\!\!\!\!\!\!\!
\sum_{\lambda \in \sigma_\mathrm{p}(A) \cap (0,\infty)} \!\!\!\!\!\!\!\!\!\!\!
\left(\kappa_-(\lambda) + \kappa_0(\lambda)\right)\, +
\sum_{i} \kappa_0(\mu_i)
\leq\kappa,
\end{equation}
and, if $0 \not\in \sigma_\mathrm{p}(A)$, then  equality holds in \eqref{ReklameSumme}.
\item [{\rm (ii)}] There are at most $\kappa$ different real nonzero eigenvalues of $A$ with 
corresponding Jordan chains of length greater than one. The length of each of these 
chains is at most $2\kappa +1$. 
\item [{\rm (iii)}] There exists a set $\Xi$ consisting of at most $\kappa$ real eigenvalues of $A$ such that all spectral points in $
(0,\infty)\setminus\Xi$ are of positive type and  all spectral points in $
(-\infty,0)\setminus\Xi$ are of negative type with respect to $A$.
\item [{\rm (iv)}] 
There exist an open neighbourhood ${\mathcal U}$ of $\ov{\mathbb R}$ in
$\ov{\mathbb C}$ and $M > 0$ such that
\begin{displaymath}
 \|(A - \lambda)^{-1}\| \leq M  \frac{(|\lambda | +1)^{4\kappa+2}}{
 |\mbox{{\rm Im}}\, \lambda |^{2\kappa+2}}
\quad 
\mbox{ for all } \lambda \in {\mathcal U} \setminus \ov{\mathbb R}.
\end{displaymath}
\item [{\rm (v)}] Let $B$ be a self-adjoint operator in
$(\mathcal K, \Skindef)$ with
$\rho(A) \cap \rho(B) \ne \emptyset$ and assume
$$
\dim\bigl(\operatorname{ran}\bigr((A-\lambda )^{-1} - (B-\lambda)^{-1}\bigr)\bigr)=n_0<\infty
$$
for some (and hence for all) $\lambda \in \rho(A) \cap \rho(B)$.
Then $B$ has $\wt{\kappa} \geq 0$ negative squares, where $\vert\wt\kappa-\kappa\vert\leq n_0$.
\end{enumerate}  
\end{theorem}

The second class of operators is the class of locally definitizable 
operators. They appeared first in
a paper by H.~Langer in 1967 \cite{L67} (without that name).
Later, P.~Jonas introduced the notion of locally definitizable
operators, see, e.g, \cite{J88,J91,J03}.

\begin{definition}\label{locdef}
Let $\Omega$ be a domain in $\ov{\mathbb C}$ which is symmetric with
respect to $\mathbb R$ such that $\Omega \cap \ov{\mathbb R} \ne
\emptyset$ and the intersections with the open upper and lower
half-plane are simply connected. Let $A$ be a self-adjoint operator
in the Krein space $({\mathcal K}, \Skindef)$ such that $\sigma(A) \cap
(\Omega \setminus \ov{\mathbb R})$ consists of isolated points which
are poles of the resolvent of $A$, and no point of $\Omega \cap
\ov{\mathbb R}$ is an accumulation point of the non-real spectrum of
$A$. The operator $A$ is called {\em definitizable over $\Omega$\/}
if the following holds.
\begin{itemize}
\item[{\rm (i)}]
For every closed subset $\Delta$ of $\Omega \cap \ov{\mathbb R}$
there exist an open neighbourhood ${\mathcal U}$ of $\Delta$ in
$\ov{\mathbb C}$ and numbers $m \geq 1$, $M > 0$ such that
\begin{displaymath}
 \|(A - \lambda)^{-1}\| \leq M \frac{(|\lambda | +1)^{2m-2}}
 {|\mbox{{\rm Im}}\, \lambda |^{m}}
 \quad 
\mbox{ for all } \lambda \in {\mathcal U} \setminus \ov{\mathbb R}.
\end{displaymath}
\item[{\rm (ii)}]
For every $\lambda \in \Omega \cap \ov{\mathbb R}$ there exists 
an open
connected neighbourhood $I_{\lambda}$ in $\ov{\mathbb R}$ and
two disjoint open intervals $I', I''$ with $I_{\lambda} \setminus \{\lambda \}=I'\cup I''$ and the following property:
All spectral points in $I'$ are either of positive or of negative type 
and all spectral points in $I''$ are either of positive or of negative type 
with respect to $A$. 

\end{itemize}
\end{definition}

Obviously, 
Theorem~\ref{Reklame} implies that a self-adjoint operator in a Krein space with finitely many negative squares is 
 definitizable over $\overline{\mathbb C}$.  
Operators definitizable over $\overline{\mathbb C}$ are definitizable in the sense of H. Langer \cite{L82} (see also  \cite{J03}),
which is a well-studied class of operators in Krein spaces.

Roughly speaking, the property of an operator to be locally definitizable
is stable under finite rank perturbations and the local sign type properties are preserved. A similar statement is valid 
for the accumulation of real discrete eigenvalues.
This is the content
of the next theorem from \cite{B07-2} and item (ii) from
\cite{BMT14}.

\begin{theorem}   \label{finite}
Let $\Omega$ be as in Definition~\ref{locdef}  and
let $A$ and $B$ be self-adjoint operators in a Krein space $({\mathcal K},\Skindef)$ with $\rho(A)\cap\rho(B)\cap\Omega \ne \emptyset$, and assume
that for some $\lambda_0 \in\rho(A)\cap\rho(B)$  the difference
\begin{equation*} 
(A-\lambda_0)^{-1} -(B-\lambda_0)^{-1}
\end{equation*}
is a finite rank operator. Then the operator $A$ is definitizable over $\Omega$
if and only if  $B$ is definitizable over $\Omega$; in this case the following holds.
\begin{itemize}

\item[\rm (i)] If $I\subset\Omega\cap\mathbb R$ is an open interval with boundary point
$\lambda\in\Omega\cap\overline{\mathbb R}$ and the spectral points in $I$ are of
positive (negative) type with respect to $A$, then there exists an open interval $I^\prime$, $I^\prime\subset I$,
with boundary point $\lambda$ such that the spectral points in $I'$ are of positive (negative, respectively) type with respect
to $B$.
\item[\rm (ii)] 
Let $J$ be an open interval such that $\overline J\subset\Omega\cap\overline{\mathbb R}$.
Then $\sigma(A)\cap J$
consists of finitely many discrete eigenvalues if and only if $\sigma(B)\cap J$ consists of finitely many discrete eigenvalues.
\end{itemize}
\end{theorem}


\end{document}